\documentclass[11pt]{article}

\usepackage[margin=1in]{geometry}

\usepackage{graphicx, transparent}      
\usepackage{tikz}          
\usetikzlibrary{arrows.meta, positioning, shapes.geometric, calc, patterns}

\usepackage{amsmath, amssymb, amsfonts, mathtools, bm, mathrsfs}

\usepackage{algorithm}
\usepackage{algorithmic}

\usepackage{subcaption}    
\usepackage{adjustbox}     
\usepackage{arydshln}      
\usepackage{enumerate}
\graphicspath{{figures/}}  

\DeclareMathOperator*{\argmin}{arg\,min}

\DeclareMathOperator{\der}{D}
\DeclareMathOperator{\T}{T}
\DeclareMathOperator{\F}{F}

\DeclareMathOperator{\nlog}{nlog}
\DeclareMathOperator{\flog}{flog}

\newcommand{\codecomment}[1]{\hfill {\footnotesize \texttt{// #1}}}

\let\origvdots\vdots
\let\origddots\ddots

\newcommand{\svdots}{\scalebox{0.6}{$\origvdots$}}
\newcommand{\sddots}{\scalebox{0.6}{$\origddots$}}

\newcommand{\sbmatrix}[1]{%
  \begingroup
  \let\vdots\svdots
  \let\ddots\sddots
  \left[\begin{smallmatrix}#1\end{smallmatrix}\right]
  \endgroup
}

\newcommand{\realmat}[1]{\mathbb{R}^{#1 \times #1}}

\newcommand{\realset}{\mathbb{R}}

\newcommand{\zerov}{\mathbf{0}}

\newcommand{\so}{\mathbb{SO}}

\newcommand{\orth}{\mathbb{O}}
\newcommand{\skewm}{\mathbf{Skew}}

\newcommand{\dexp}{\der\exp}
\newcommand{\dexpof}[1]{\dexp\left(#1\right)}

\newcommand{\outerprodskew}[2]{\sum_{i=1}^m #1_{[i]}\big[\begin{smallmatrix}0 & -#2_i\\#2_i & 0\end{smallmatrix}\big]#1_{[i]}^{\T}}
\newcommand{\outerprodshift}[3]{\sum_{i=1}^m #1_{[i]}\big[\begin{smallmatrix}0 & -#2_i-2\pi #3_i\\#2_i+2\pi #3_i & 0\end{smallmatrix}\big]#1_{[i]}^{\T}}

\newcommand{\prng}{(-\pi,\pi]}

\newcommand{\cale}{\mathcal{E}}

\newcommand{\calz}{\mathcal{Z}}

\newcommand{\frakq}{\mathfrak{Q}}

\newcommand{\fraks}{\mathfrak{S}}

\usepackage{amsmath, amssymb, amssymb, amsfonts, amsthm}
\usepackage{mathrsfs}
\usepackage{color}
\usepackage{verbatim}
\usepackage{subcaption}

\usepackage[
    colorlinks=true,
    linkcolor=blue,
    citecolor=blue,
    urlcolor=blue
]{hyperref}

\usepackage[capitalize, nameinlink]{cleveref}

\numberwithin{equation}{section}

\theoremstyle{plain}
\newtheorem{theorem}{Theorem}[section]

\newtheorem{definition}[theorem]{Definition}
\newtheorem{lemma}[theorem]{Lemma}
\newtheorem{proposition}[theorem]{Proposition}
\newtheorem{corollary}[theorem]{Corollary}
\newtheorem{remark}[theorem]{Remark}

\crefname{theorem}{Theorem}{Theorems}
\Crefname{theorem}{Theorem}{Theorems}

\crefname{proposition}{Proposition}{Propositions}
\Crefname{proposition}{Proposition}{Propositions}

\crefname{definition}{Definition}{Definitions}
\Crefname{definition}{Definition}{Definitions}

\crefname{lemma}{Lemma}{Lemmas}
\Crefname{lemma}{Lemma}{Lemmas}

\crefname{corollary}{Corollary}{Corollaries}
\Crefname{corollary}{Corollary}{Corollaries}

\crefname{remark}{Remark}{Remarks}
\Crefname{remark}{Remark}{Remarks}

\usepackage{algorithm}
\usepackage{algorithmic}

\newcommand{\logVenn}[1][0.5]{
  \begin{tikzpicture}[scale=#1, every node/.style={transform shape}, font=\large]
        \node[draw, line width=1.5pt, minimum width=6.8cm, minimum height=4cm] (A) at (0,0) {};
        \node[anchor=south west, inner sep=2mm] at (A.south west) {Frobenius-Nearest Preimage(s) to $S$};

        \node[draw, dashed, minimum width=2.8cm, minimum height=3.3cm, anchor=north east] (D) at ([xshift=-1.2cm]A.north east) {};
        \node[anchor=south west, inner sep=2mm] at (D.south west) {$\|S-X\|_2 <\pi$};

        \node[draw, line width=1.5pt, minimum width=6.5cm, minimum height=4cm, anchor=south west] (B) at ([yshift=0.8cm]D.south west) {};
        \node[anchor=north west, inner sep=2mm] at (B.north west) {Preimage(s) D-Connected with $S$};

        \fill[pattern=north east lines, pattern color=gray] (B.south west) rectangle (D.north east); 
        
        \node[align=center] at ($(B.south west)!0.5!(D.north east)!0.5!(D.center)$) {Nearby\\ Logarithm};

        \coordinate (DiffNE) at ($(B.north east)+(-1.0cm,-0.8cm)$);

        \fill[pattern=north west lines, pattern color=gray] 
            (B.south west) rectangle (DiffNE);
        \draw (B.south west) rectangle (DiffNE);

        \node[anchor=north west, inner sep=2mm]
            at ($(B.south west |- DiffNE)$)
            {Diffeomorphic Logarithm};

        \node[
            draw,
            line width=1.2pt,
            minimum width=10.2cm,
            minimum height=6.4cm,
            anchor=center
        ] (G) at ($(A.center)!0.5!(B.center)$) {};

        \node[anchor=north west, inner sep=2mm]
            at (G.north west)
            {Preimages of $Q$};
    \end{tikzpicture}
}

\title{Diffeomorphic Logarithm of Special Orthogonal Matrices\thanks{This work was supported by the National Natural Science Foundation of China (No. 12371311), the Natural Science Foundation of Fujian Province (No. 2023J06004), and the Fonds de la Recherche Scientifique -- FNRS under Grant no T.0001.23. This work was done in part when Kyle A. Gallivan was visiting UCLouvain, supported by the ``Professeurs et chercheurs visiteurs'' budget of the Science and Technology Sector.}}

\author{Zhifeng Deng\thanks{School of Mathematical Sciences, Xiamen University, Xiamen, China
  (\url{zhifengdeng@xmu.edu.cn}).}
\and P.-A. Absil\thanks{ICTEAM Institute, UCLouvain, Louvain-la-Neuve, Belgium
  (\url{pa.absil@uclouvain.be}).}
\and Kyle A. Gallivan\thanks{Department of Mathematics, Florida State University, Tallahassee, USA (\url{kgallivan@fsu.edu}).}
\and Wen Huang\thanks{Corresponding author. School of Mathematical Sciences, Xiamen University, Xiamen, China
(\url{wen.huang@xmu.edu.cn}).}
}

\begin{document}

\maketitle

\abstract{
    The special orthogonal group $\so_n$ is a Lie group whose geometry and local structure are encoded by the exponential map in its Lie algebra $\skewm_n$, the set of skew-symmetric matrices. The associated multi-valued inverse problem---the matrix logarithm---in $\so_n$ exhibits a highly nontrivial local diffeomorphism structure, which differs from the matrix logarithm for invertible matrices. This work characterizes the local diffeomorphism structure of the exponential in the set of skew-symmetric matrices where its derivative is invertible. We show that this set with an invertible derivative can be organized into diffeomorphic regions, using a canonical alignment of Schur decompositions. In particular, the region that contains the principal logarithm has a special multiplicity structure: each matrix in $\so_n$ admits at most two skew-symmetric preimages in this region. Based on this geometric framework, we introduce the diffeomorphic logarithm of special orthogonal matrices together with an efficient and stable algorithm. Moreover, it is applied to the Karcher mean problem in $\so_n$, demonstrating continuous behavior of the mean under perturbations of the data, which is not captured by the principal logarithm.
}

\vspace{0.5cm}


\vspace{0.5cm}


\section{Introduction}

The matrix exponential
\begin{equation*}
    \begin{aligned}
        \exp:\,&\skewm_n := \{X \in \realmat{n} : X + X^{\T} = 0\}\\
        &\to \so_n:= \{Q \in \realmat{n} : Q^{\T}Q = I_n,\ \det(Q)=1\},
    \end{aligned}
\end{equation*}
maps the space of real skew-symmetric matrices to the special orthogonal group and encodes the local geometry of~$\so_n$ through its Lie algebra. Beyond its geometric role, the exponential and its inverse are fundamental primitives in numerical algorithms in matrix manifolds, including the computation of geodesics and solving optimization methods in the Stiefel and orthogonal manifolds, e.g., the computation of geodesics in the Stiefel manifolds~\cite{mataigne2025stiefel} and the Karcher mean problem in $\so_n$~\cite{moakher2002means}.

The exponential $\exp: \skewm_n \to \so_n$ is surjective, but a central difficulty is that it is not injective; see, e.g.,~\cite{zhou2019continuity}, which proposes a continuation strategy in $\so_3$ to avoid branch selection when computing the matrix logarithm. In practice, a single-valued inverse is obtained by selecting a branch that restricts eigenvalues of a preimage to a fixed range. The most common choice is the principal branch, where eigenvalues have imaginary parts in $(-\pi,\pi)$; this defines the principal logarithm. However, this selection becomes unstable (see, e.g.,~\cite{dieci1999real}) near the boundary of the branch, namely when rotations in $\so_n$ have angles close to $\pi$ in some directions (i.e., near an antipodal point). Consequently, there exist continuous curves $S(t)$ in $\skewm_n$ where the principal logarithm of $\exp(S(t))$ fails to be continuous. For example, the Euler-angle representations in $\so_3$ exhibit discontinuity and branch-switching phenomena near rotations of angle $\pi$, causing well-known issues in robotics and attitude control; see, e.g.,~\cite{diebel2006representing}.

Moreover, continuity alone is insufficient for many applications. In optimization problems, such as the $\so_n$ Karcher mean~\cite{moakher2002means}, local diffeomorphism plays a crucial role: it ensures that an objective $f$ defined in $\so_n$ can be equivalently represented in $\skewm_n$ via the pullback of the objective
\begin{equation*}
    \hat{f}(X) = f(\exp(X)),
\end{equation*}
with the local optimization geometry of $\hat{f}$ around $X$ preserving the critical points and their nature. This equivalence relies on the invertibility of the derivative $\dexpof{X}$. For general real matrices, the differential of the matrix exponential is invertible throughout the principal branch, but becomes rank-deficient on its boundary. Consequently, pullback-based optimization formulations via $\exp$ in $\realmat{n}$ are naturally limited to the principal branch, for which the principal logarithm is sufficient. In contrast, for skew-symmetric matrices, the local diffeomorphism structure of $\exp$ extends beyond the principal branch. This distinction enables optimization problems in $\so_n$ to be formulated in a larger domain in $\skewm_n$, providing a foundation for further analysis and suggesting improved practical performance. For example, in the $\so_n$ Karcher mean problem, we observe more consistent behavior under perturbations of the data.

Recent work in~\cite{deng2025expskew} provides a structural explanation for this phenomenon by analyzing the differential of the exponential restricted to~$\skewm_n$. In particular, the set of points where the differential is rank-deficient,
\begin{equation}\label{eq:tangent-conjugate-locus-def}
    \fraks := \{X \in \skewm_n : \dexpof{X} : \skewm_n \to T_{\exp(X)}\so_n \text{ is rank-deficient}\},
\end{equation}
was characterized, revealing that the complement $\skewm_n \setminus \fraks$ decomposes into countably many open connected components consisting entirely of points with invertible differential. Notably, the principal branch (restricted to $\skewm_n$) is a proper subset of a component denoted by $\mathcal{C}_*$. As a result, one may move beyond the boundary of the principal branch within $\mathcal{C}_*$ without encountering rank-deficiency of $\dexp$, thereby preserving a local diffeomorphism of~$\exp$ on suitable neighborhoods while the principal logarithm ceases to describe a smooth continuation. This discovery shows that the local diffeomorphism structure of $\exp:\skewm_n \to \so_n$ is substantially richer than what is obtained by simply restricting the exponential from~$\mathbb{C}^{n\times n}$ to~$\skewm_n$. Consequently, it enables the notion of a \emph{nearby logarithm} as a local inverse of~$\exp$ around a reference point. However, the geometric structure of these components, the existence and selection of inverse images for a given special orthogonal matrix, and the design of practical algorithms remain largely open.

The goal of this work is to establish a complete and constructive
diffeomorphic logarithm theory on~$\so_n$ based on the component
structure of~$\skewm_n \setminus \fraks$.
We analyze the exponential map in each component,
characterize the multiplicity and geometry of its inverse images,
and develop an efficient algorithm that produces logarithms
compatible with smooth continuation.

\subsection{Related work}

The matrix logarithm and its inverse relationship with the exponential have been extensively studied in matrix analysis. Classical results establish the existence of real logarithms under spectral conditions and identify the principal logarithm as the unique inverse obtained from the principal branch of the exponential (see, e.g.,~\cite[Chapter 11]{higham2008}), on which the exponential is a diffeomorphism. Numerical algorithms for computing matrix logarithms, including Schur–Parlett recursions~\cite{davies2003schurparlett}, inverse scaling–squaring methods~\cite{al2013logarithm} for general invertible matrices, and approaches specialized to special orthogonal matrices~\cite{gallier2003computing,cardoso2010exponentials}, likewise compute a single-valued branch compatible with the principal domain, rather than characterizing alternative regions where the exponential remains locally diffeomorphic. These existing algorithms typically rely on eigenvalue decompositions, angle unwrapping, or principal-branch constructions, and therefore inherit the discontinuities associated with the boundary of the principal domain. Such approaches do not explicitly address the component-wise diffeomorphic structure of the exponential that is required for smooth continuation in geometric numerical methods.

The continuity of rotation representations has been investigated in the literature, especially through quaternions for $\so_3$. Quaternions for $\so_3$ represent rotations continuously over angles in $(-2\pi,2\pi)$ (beyond the principal angle range $(-\pi,\pi)$), and have been widely used in robotics, computer graphics, computer vision, and attitude control; see, e.g.,~\cite{diebel2006representing}. This quaternion characterization is naturally subsumed by the local diffeomorphic structure developed in this paper; see~\cref{subsec:so3}. Another continuous representation of rotations is by relaxing the first $n-1$ columns of $Q\in\so_n$ to a full-rank matrix $U\in\mathbb{R}^{n\times (n-1)}$ in the ambient space, see~\cite{zhou2019continuity}. The corresponding rotation is then recovered from $U$ by the Gram--Schmidt process. This characterization is not unique, and its differential structure is governed by the complicated re-orthogonalization map, which also relies on algorithmic choices. This is natural in the learning setting considered in~\cite{zhou2019continuity}, where the geometry of the objective is learned from sampled data. In contrast, the present work investigates the local diffeomorphism structure of $\exp:\skewm_n\to\so_n$, revealing the exact logarithmic geometry.

From the Lie-theoretic viewpoint (see, e.g.,~\cite[Sec.~1.2]{rossmann2006lie}), $\so_n$ is a smooth matrix Lie group whose exponential map is locally diffeomorphic away from points where the differential becomes singular. While the rank-deficiency of the differential is well-understood in general Lie theory, explicit structural descriptions within the skew-symmetric algebra~$\skewm_n$ are largely absent from the classical literature.

From the normal matrices viewpoint, the work of Mataigne and Gallivan~\cite{mataigne2024eigenvalue} provides an essential theoretical foundation, namely the invariance group of Schur bases of normal matrices. Specifically, for a Schur decomposition of a normal matrix $M = RTR^{\T}$ where $R\in \orth_n$ is a Schur basis and $T$ is block-diagonal, the invariance group $\{O\in \orth_n: T = O^{\T} TO\}$ describes all equivalent Schur bases $RO$ in the transformed decomposition $M = (RO)(O^{\T}TO)(RO)^{\T}$. Building on this structural insight, Deng et al.~\cite{deng2025expskew} establishes a relationship between the Schur bases of skew-symmetric matrices and special orthogonal matrices (reviewed in~\cref{prop:shared-basis}). In contrast to the invariance characterization for one normal matrix, this result reveals a structural correspondence between two geometrically linked matrices, thereby enabling a further exploitation of the geometry of $\so_n$ in the present work. The efficient Schur decomposition algorithm for normal matrices proposed in~\cite[Algorithm 3.1]{mataigne2024eigenvalue} also provides a fundamental computational framework for logarithm algorithms, which enables stable and efficient decompositions for skew-symmetric and special orthogonal matrices, forming a key algorithmic building block for the logarithm and inverse-exponential procedures studied here.

Most closely related to this paper is the recent analysis in~\cite{deng2025expskew}, which characterized the rank-deficient set $\fraks$ of the differential of the exponential restricted to~$\skewm_n$ and showed that its complement decomposes into countably many connected components of points with invertible differential. That work introduced the notion of a nearby logarithm as a local inverse defined within a suitable neighborhood of a reference point. However, the global geometry of these components, the multiplicity and selection of inverse images for a given orthogonal matrix, and the construction of practical algorithms compatible with smooth continuation were left unresolved. The present paper provides a complete component-wise diffeomorphic logarithm framework together with explicit inverse characterizations and an efficient Schur-based computational method.

\subsection{Contributions}

The main contributions of this work are as follows.

\begin{enumerate}
    \item
    A description of the connected components of $\skewm_n \setminus \fraks$. Moreover, the preimage by exp is shown to be a singleton for all points in the non-special component, and a doubleton for all points in the special component $\mathcal{C}_*$ containing the principal branch.

    \item
    The definition of a \emph{diffeomorphic logarithm} based on the connected components of $\skewm_n \setminus \fraks$, extending the nearby logarithm of~\cite{deng2025expskew} and clarifying its relationship with the Frobenius-nearest preimage in $\skewm_n\setminus\fraks$; see~\cref{fig:venn} and~\cref{remark:venn}.

    \item
    An efficient and robust algorithm for computing the proposed diffeomorphic logarithm, or a fallback preimage in the closure of any component when the diffeomorphic logarithm is undefined, with computational cost of two Schur decompositions and one matrix multiplication.
\end{enumerate}

\subsection{Organization}

The remainder of this paper is organized as follows.
\Cref{sec:preliminaries} reviews preliminary results and introduces the
notation used throughout.
The two-preimage structure in the special component~$\mathcal{C}_*$
and the diffeomorphic behavior on the remaining components are
established in \cref{sec:tangent-conjugate-complement}.
\Cref{sec:conjugate-complement} introduces the notion of a canonical
Schur decomposition, and identifies a key invariant of the countably many
components. \Cref{sec:diffeomorphic-logarithm} defines the diffeomorphic logarithm and its guaranteed fallback, and presents the algorithms for computing them. \Cref{sec:experiments} presents numerical experiments comparing
computational cost with the principal logarithm in \texttt{MATLAB} and
demonstrating performance in Karcher mean computations on~$\so_n$. Conclusions are drawn in~\cref{sec:conclusion}.

\section{Preliminaries}\label{sec:preliminaries}

\subsection{Matrices of interests}

Let $\skewm_n:=\{X\in\realmat{n}:X+X^{\T} = \zerov\}$ denote the set of $n\times n$ skew-symmetric matrices and $\so_n := \{X\in\realmat{n}:Q^{\T}Q = I_n, \det(Q) = 1\}$ denote the set of $n\times n$ special orthogonal matrices. The two sets are related by the matrix exponential map: $\exp:\skewm_n\to \so_n$. Since the matrix exponential is not a bijection, the ``inverse'' of $\exp$ refers to the preimage
\begin{equation*}
    \exp^{-1}(Q):=\{X\in\skewm_n:\exp(X) = Q\}, \forall Q\in\so_n.
\end{equation*}

The differential of the exponential restricted to $\skewm_n$ is denoted as 
\begin{equation*}
    \dexpof{S}:\skewm_n\to T_{\exp(S)}\so_n, X\mapsto \dexpof{S}[X] = \exp(S)Y
\end{equation*}
where $X$ and $Y$ are skew-symmetric matrices. Unless otherwise specified, the notation $\dexpof{S}$ always refers to the differentiation of the matrix exponential on \emph{skew-symmetric matrices} at the \emph{skew-symmetric} $S$. 

The invertibility of $\dexpof{S}$ is one of the most essential properties throughout this paper. The tangent conjugate locus $\fraks$ and the D-notation in~\cref{def:d-notation} are introduced for simplicity. 

\begin{definition}\label{def:conjugate-locus}
    The \emph{tangent conjugate locus}\footnote{The term \emph{tangent conjugate locus} comes from differential geometry; see, e.g.,~\cite[\S 5.3]{dC92}.} $\fraks$~\eqref{eq:tangent-conjugate-locus-def} denotes the set of $X\in\skewm_n$ with rank-deficient $\dexpof{X}:\skewm_n\to T_{\exp(X)}\so_n$, 
    and its image under the matrix exponential is termed the \emph{conjugate locus} in $\so_n$, denoted by $\frakq:= \{\exp(X):X\in\fraks\}$. Moreover, the respective complements are termed the \emph{tangent conjugate complement} $\skewm_n\setminus\fraks$ and the \emph{conjugate complement} $\so_n\setminus \frakq$.
\end{definition}

\begin{definition}\label{def:d-notation}
    A \emph{D-curve} is a continuous curve in the tangent conjugate complement $\skewm_n\setminus\fraks$, denoted by $\Gamma:[0,1]\to\skewm_n\setminus\fraks$. Moreover, two matrices in $\skewm_n\setminus\fraks$ are termed \emph{D-connected} if they are connected by a D-curve, and a maximal path-connected subset of $\skewm_n\setminus\fraks$ is termed a \emph{D-component}.
\end{definition}

The \emph{nearby logarithm} proposed in~\cite{deng2025expskew} is then reviewed in~\cref{def:nearby-logarithm}.

\begin{definition}\label{def:nearby-logarithm}
    Given a skew-symmetric matrix $S\in\skewm_n\setminus\fraks$, consider the neighborhood
    \begin{equation*}
        \mathcal{N}_S:=\{X\in\skewm_n\setminus\fraks: \text{$S$ and $X$ are D-connected and $\|X-S\|_2<\pi$}\},
    \end{equation*}
    where $\|\cdot\|_2$ denotes the matrix $2$-norm (i.e., the spectral norm). Then, $\exp:\mathcal{N}_S\to\calz_S:=\{\exp(X):X\in\mathcal{N}_S\}$ is a diffeomorphism and its inverse is termed the \emph{nearby logarithm around $S$}, denoted by
    \begin{equation}\label{eq:nearby-logarithm}
        \nlog_S:\calz_S\to\mathcal{N}_S.
    \end{equation}
\end{definition}

\subsection{Real Schur Decomposition}\label{subsec:notations}

A real Schur decomposition $X = R\,TR^{\T}$ of $X\in \realmat{n}$ finds an orthogonal Schur basis $R\in\orth_n:=\{Q\in\realmat{n}:Q^{\T}Q = I_n\}$ and a block upper-triangular matrix $T$ where the blocks on the diagonal are either $1\times 1$ or $2\times 2$. Moreover, for a normal matrix $X$ (i.e., $X^{\T}X = XX^{\T}$), the block upper-triangular matrix $T$ becomes block diagonal. Consequently, for both skew-symmetric matrices and special orthogonal matrices, the diagonal blocks in $T$ can be ordered and grouped into blocks with size of $2\times 2$---except a single $1\times 1$ leftover block when $n$ is odd---via permuting the column orders in the Schur basis $R$. We introduce the integers $m,k\in\mathbb{Z}$ to distinguish the parity of $n$: (i) when $n$ is even, $2m = 2k = n$, or (ii) when $n$ is odd, $2m+1 = 2k-1 = n$. Under the $m,k$ convention, we introduce the following block partition on $n\times n$ matrices and the respective bracket index notation to express the special structure of Schur decomposition on normal matrices.

For an $n\times n$ matrix $X$, the subscript with one bracketed index $X_{[i]}, i \leq m$ denotes the $i$-th pair of two columns in $X$. When $n = 2m+1$ is odd, $X_{[k]}, k = m+1$ denotes the left-over column. Similarly, the subscript with two bracketed indices $X_{[i,j]}, i,j\leq m$ denotes the $i,j$-th $2\times 2$ block in $X$ while the $X_{[i,k]}, X_{[k,j]}$ and $X_{[k,k]}$ with $i,j\leq m, k = m+1$ denote the blocks in the left-over row, column, and diagonal. When $n = 5, m= 2, k =3$, the block partitions are given by
\begin{equation*}
    \begin{aligned}
        \left[\begin{array}{cc:cc:c}
            X_{11} & X_{12} & X_{13} & X_{14} & X_{15}\\
            X_{21} & X_{22} & X_{23} & X_{24} & X_{25}\\
            \hline
            X_{31} & X_{32} & X_{33} & X_{34} & X_{35}\\
            X_{41} & X_{42} & X_{43} & X_{44} & X_{45}\\
            \hline
            X_{51} & X_{52} & X_{53} & X_{54} & X_{55}\\
        \end{array}\right] \begin{aligned}
            &=\left[\begin{array}{c:c:c}
                X_{[1,1]} & X_{[1, 2]} & X_{[1, 3]}\\
                \hline
                X_{[2,1]} & X_{[2, 2]} & X_{[2, 3]}\\
                \hline
                X_{[3,1]} & X_{[3, 2]} & X_{[3, 3]}\\
            \end{array}\right]\text{ in blocks,}\\
            &\text{ or }=\left[\begin{array}{c:c:c}
                X_{[1]} & X_{[2]} & X_{[3]}\\
            \end{array}\right] \text{ in columns.}
        \end{aligned}\\
    \end{aligned}
\end{equation*}

The Schur decomposition of normal matrix, $X = RTR^{\T}$, discussed in this paper is then assumed to be in the form of 
\begin{equation}\label{eq:block-diagonal-normal}
    X = \sum_{i=1}^k R_{[i]}T_{[i,i]}R_{[i]}^{\T},\, \text{where}\,\begin{cases}
        T_{[i,j]} = \zerov, & i\neq j,\\
        T_{[i,i]} = \sbmatrix{r_i & -s_i\\ s_i & r_i}, & i\leq m,\\
        T_{[k,k]} = r_k, & k = m+1.
    \end{cases}
\end{equation}
Note that $T_{[i,i]}\in \realmat{2}$, for $i\leq m$, is a normal matrix because $X$ is normal, which yields the expression of $2\times 2$ normal matrix $\sbmatrix{r & -s\\ s & r}$ for some $r,s\in\realset$.

\textbf{Skew-symmetric matrices:} When $X$ is skew-symmetric, the block diagonal $T$ is also skew-symmetric. Consequently, each $2\times 2$ skew-symmetric block is of the form $T_{[i,i]} = \sbmatrix{0 & -\theta_i\\\theta_i & 0}$ with $\theta_i\in\realset$ for $i\leq m$, and $\sbmatrix{0}$ is the only $1\times 1$ skew-symmetric block.

\textbf{Special orthogonal matrices:} For any special orthogonal $Q$, we reserve the letter $E$ as the special orthogonal block diagonal $T$ (to distinguish it from the skew-symmetric case), i.e., $Q = RER^{\T}$. Similarly, each $2\times 2$ special orthogonal block is of the form $E_{[i,i]} = \sbmatrix{r_i & -s_i \\s_i & r_i}$ where $r_i^2 + s_i^2 = 1$ for $i\leq m$, and $\sbmatrix{1}$ is the only $1\times 1$ special orthogonal block.

\textbf{Auxiliary angle:} While the technical details and some of the statements differ in the instances of even $n = 2m = 2k$ or odd $n = 2m+1 = 2k-1$, introducing an auxiliary zero angle as $\theta_k \equiv 0$ for $n = 2k-1$ is convenient for unifying some analyses and statements.

\begin{definition}\label{def:angles-and-shifts}
    This paper uses the following notation:
    \begin{itemize}
        \item \textbf{Dimension}: For $n\times n$ matrices, $m =\lfloor n/2\rfloor$ and $k = \lceil n/2\rceil$ are in use.  
        \item \textbf{Schur Basis}: The letters $R$ and $V$ are reserved for orthogonal Schur bases.
        \item \textbf{Angles}: The letters $\alpha$ and $\beta$ are reserved for vectors of angles in $\realset^k$. Moreover, if a vector lies in $(-\pi,\pi]^k$, it is termed \emph{a vector of principal angles}, denoted by the letter $\theta$ or $\sigma$. When $n = 2k-1$, the $k$-th angle is an auxiliary angle set to be $0$. 
        \item \textbf{Shifts}: The letters $\xi:=(x_1,\ldots, x_k)$ and $\eta:=(y_1,\ldots, y_k)$ are reserved for vectors of integers in $\mathbb{Z}^k$ between two vectors of angles (e.g., $\alpha = \theta+2\pi\xi$).  When $n = 2k-1$, the $k$-th shift is an auxiliary shift set to be $0$. 
        \item \textbf{Block diagonal notation}: The capitalized letters of angles are reserved for the block diagonal matrix containing the respective angles in the appropriate dimension, e.g.,
        \begin{equation*}
            \Theta = \begin{cases}
                {\rm diag}\big(\sbmatrix{0 & -\theta_1\\\theta_1 & 0}, \ldots, \sbmatrix{0 & -\theta_m\\\theta_m & 0}\big), &n = 2m = 2k,\\
                {\rm diag}\big(\sbmatrix{0 & -\theta_1\\\theta_1 & 0}, \ldots, \sbmatrix{0 & -\theta_m\\\theta_m & 0}, 0\big), &n = 2m+1 = 2k-1.\\
            \end{cases}
        \end{equation*}
        \item \textbf{Skew-symmetric notation}: A skew-symmetric matrix $S$ with the angles $\theta$ under the Schur basis $R$, as an example, is given by
        \begin{equation*}
            S = R \Theta R^{\T} = \textstyle \sum_{i=1}^m R_{[i]}\sbmatrix{0 & -\theta_i\\\theta_i & 0}R_{[i]}^{\T}.
        \end{equation*} 
        \item \textbf{Special orthogonal notation}: A special orthogonal matrix $Q$ under the Schur basis $R$, as an example, is given by
        \begin{equation*}
            Q = \begin{cases}
                \textstyle \sum_{i=1}^m R_{[i]}\sbmatrix{r_i & -s_i\\s_i & r_i}R_{[i]}^{\T} & n = 2m = 2k,\\[4pt]
                \textstyle \sum_{i=1}^m R_{[i]}\sbmatrix{r_i & -s_i\\s_i & r_i}R_{[i]}^{\T} + R_{[k]}R_{[k]}^{\T} & n = 2m+1 = 2k-1,\\
            \end{cases}
        \end{equation*}
        where $r_i^2  + s_i^2 = 1$ holds for all $i\leq m$. 
    \end{itemize}
\end{definition}

\section{Tangent Conjugate Complement}\label{sec:tangent-conjugate-complement}

This section establishes the diffeomorphism structure of the matrix exponential in the complement $\skewm_n\setminus\fraks$. This complement is decomposed into connected and open D-components $\mathcal{C}_e, e\in\cale$ where $\mathcal{E}$ is a countable set of indices; see~\cite[Corollary~4.4]{deng2025expskew}. The characterization of the index set $\cale$ is given in~\cref{sec:conjugate-complement}. 

\Cref{fig:components} illustrates the D-components of $\skewm_5\setminus\fraks$ by the two non-auxiliary angles $\theta_1$ and $\theta_2$. Observe that the principal logarithm of special orthogonal matrices is strictly contained in a D-component denoted by
\begin{equation}\label{eq:star-component}
    \mathcal{C}_*=\big\{\textstyle\sum_{i=1}^mR_{[i]}\sbmatrix{0 & -\theta_i\\\theta_i & 0}R_{[i]}^{\T}:
    R\in \orth_n, |\theta_i\pm\theta_j|<2\pi,\forall\, i\neq j\leq k\big\},
\end{equation}
Recall that $k = \lceil n/2\rceil$, $m = \lfloor n/2 \rfloor$, and $R_{[i]} = \sbmatrix{R_{2i-1} & R_{2i}}$ refers to the $(2i-1)$th and $(2i)$th columns in $R$, see details in~\cref{subsec:notations}.
\begin{figure}[tbp]
    \centering
    \includegraphics[width=.6\textwidth]{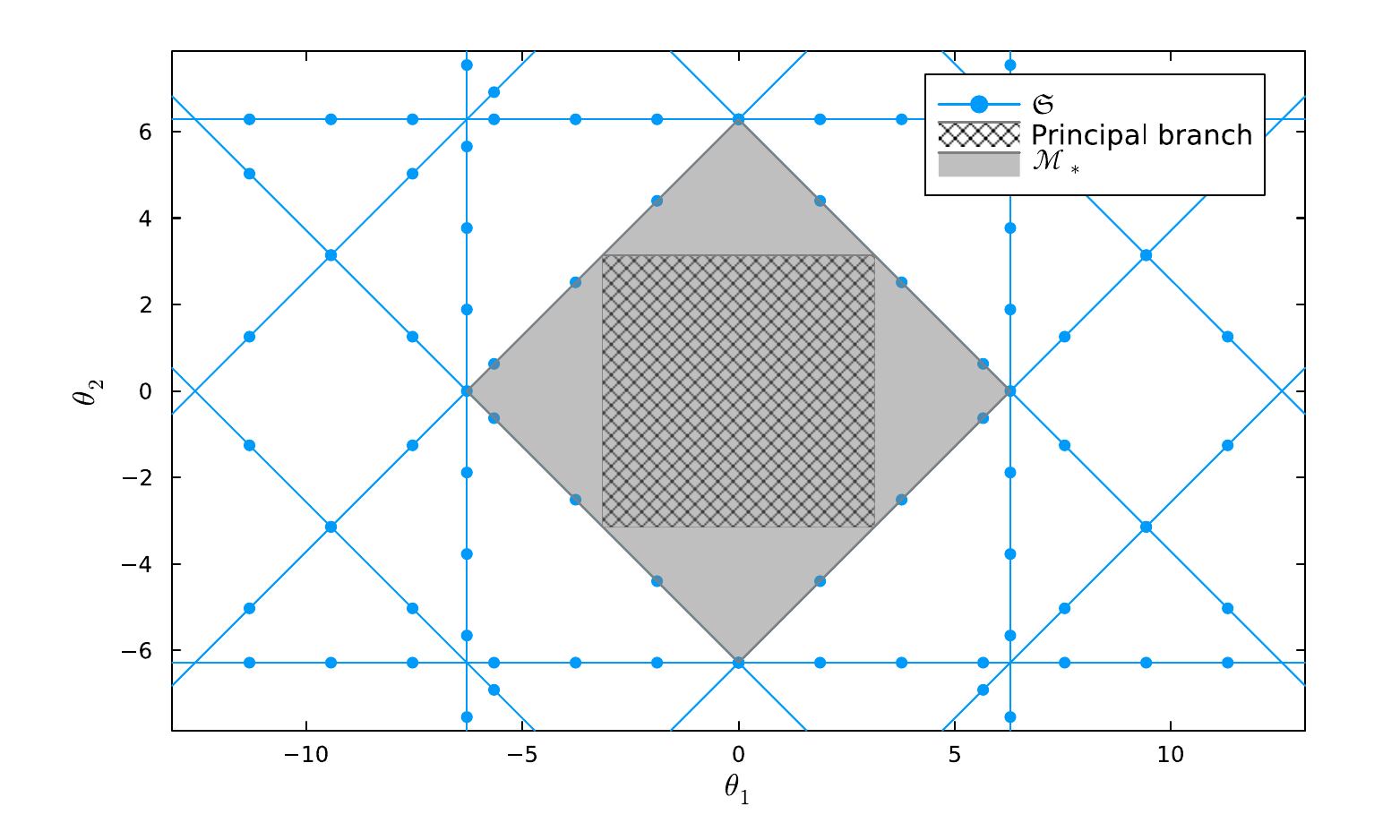}
    \caption{D-components of $\skewm_5\setminus\fraks$ illustrated by the non-auxiliary angle.\\ The vertical and horizontal lines in $\fraks$ are due to the presence of the auxiliary angle $\theta_3 = 0$ in view of~\cref{prop:tangent-conjugate-locus}. When $n=4$, those lines disappear and the rest of the illustration is unchanged.}\label{fig:components}
\end{figure}

In view of the inverse function theorem, $f:\mathcal{X}\to\mathcal{Y}$ is a diffeomorphism if and only if the derivative $\der f(x)$ in $T_x\mathcal{X}$ is invertible for all $x\in\mathcal{X}$, and $f$ is smooth and bijective in $\mathcal{X}$. For a D-component $\mathcal{C}_e$, the invertible $\dexpof{X}$ for all $X\in\mathcal{C}_e$ follows immediately. To investigate the bijective properties of $\exp$ in D-components, this section investigates the distribution of the preimages $\exp^{-1}(Q)$ among the D-components $\mathcal{C}_e, e\in\cale$. In particular, the following statements hold:
\begin{enumerate}[(i)]
    \item $\mathcal{C}_*\cap \exp^{-1}(Q)$ has at most two preimages, as characterized in~\cref{prop:dconnected-dpreimage};
    \item $\mathcal{C}_e\cap \exp^{-1}(Q)$ for $\mathcal{C}_e\neq \mathcal{C}_*$ has at most one preimage, leading to the diffeomorphism stated in~\cref{thm:component-diffeomorphism}.
\end{enumerate}

\subsection{Shifted Angles in D-Preimages}

This subsection identifies the preimages of a given $Q\in\so_n$ in $\skewm_n\setminus\fraks$ via shifting the principal angles $\theta\in\prng^k$ uniquely defined under a Schur basis $R$ of $Q$. Recall that the Schur decomposition of $Q = RER^{\T}\in\so_n$ consists of the diagonal blocks in the form of $E_{[i,i]} = \sbmatrix{r_i & -s_i\\ s_i & r_i}$ where $r_i^2 + c_i^2 = 1$, see~\cref{def:angles-and-shifts}. Consequently, $E_{[i,i]}$ uniquely identifies an angle $\theta_i$ such that
\begin{equation*}
    \exp\big(\sbmatrix{0 & -\theta_i\\ \theta_i & 0}\big) = \sbmatrix{\cos(\theta_i) & -\sin(\theta_i)\\ \sin(\theta_i) & \cos(\theta_i)} = \sbmatrix{r_i & -s_i\\ s_i & r_i},\quad\text{subject to } \theta_i\in \prng.
\end{equation*}
Thus, a Schur decomposition $Q=RER^{\T}$ defines a unique vector of principal angles $\theta\in\prng^k$ as summarized in~\cref{def:principal-angles}, which builds upon~\cref{def:angles-and-shifts}. 

\begin{definition}\label{def:principal-angles}
    For a Schur decomposition $Q=RER^{\T}\in\so_n$, the vector of principal angles $\theta\in(-\pi,\pi]^k$ of $Q$ under $R$ is given by
    \begin{equation*}
        \begin{cases}
            E_{[i,i]}=\sbmatrix{r_i & -s_i\\ s_i & r_i}=\big[\begin{smallmatrix}
                \cos(\theta_i) & -\sin(\theta_i)\\ \sin(\theta_i) & \cos(\theta_i)
            \end{smallmatrix}\big],&\text{for }i\leq m,\\
            \text{the auxiliary }\theta_k = 0,&\text{when }n = 2m+1=2k-1.
        \end{cases}
    \end{equation*} 
\end{definition}

\begin{remark}\label{remark:angles_symmetry}
    For $Q\in\so_n$, all the vectors of principal angles of $Q$ are symmetric with respect to permutation and sign flips of non-$\pi$ angles. Indeed, for a permutation $P:(1,\ldots, k)\mapsto (i_1,\ldots, i_k)$ and the principal angles $\theta$ under $R$, the permuted angles
    \begin{equation*}
        \hat{\theta}=(\theta_{i_1},\ldots, \theta_{i_k})\in \prng^k
    \end{equation*}
    can be obtained from the permuted outer product form of the Schur decomposition
    \begin{equation*}
        \textstyle Q = \sum_{i=1}^k R_{[i]}E_{[i,i]}R_{[i]}^{\T} = \sum_{j=1}^k R_{[i_j]}E_{[i_j,i_j]}R_{[i_j]}^{\T},
    \end{equation*}
    i.e., $\hat{\theta}$ is the principal angles of $Q$ under $\hat{R} = \sbmatrix{R_{[i_1]} & \cdots & R_{[i_k]}}$. Similarly, a flipped principal angle $-\theta_i\in (-\pi, \pi)$ can be uniquely determined from $E_{[i,i]}^{\T} = \sbmatrix{r_i & s_i\\-s_i & r_i}$ in
    \begin{equation*}
        \textstyle R_{[i]}E_{[i,i]}R_{[i]}^{\T} = \sbmatrix{R_{2i-1} & R_{2i}} E_{[i,i]}\sbmatrix{R_{2i-1} & R_{2i}}^{\T} = \sbmatrix{R_{2i} & R_{2i-1}} E_{[i,i]}^{\T}\sbmatrix{R_{2i} & R_{2i-1}}^{\T}.
    \end{equation*}
\end{remark}

Since every $Q$ in $\so_n$ admits a unique vector of principal angles under a Schur basis, the characterization of conjugate locus $\frakq$ is obtained from the characterization of the tangent conjugate locus, which is reviewed in~\cref{prop:tangent-conjugate-locus}.

\begin{proposition}\cite[Theorem 3.8]{deng2025expskew}\label{prop:tangent-conjugate-locus}
    Let $\theta\in\realset^k$ (the auxiliary $\theta_k = 0$ when $n = 2k-1$) be a vector of angles of $X\in\skewm_n$ under a Schur basis $R\in\orth_n$, i.e., $X=\sum_{i=1}^m R_{i}\big[\begin{smallmatrix}
        0 & -\theta_i\\\theta_i & 0
    \end{smallmatrix}\big]R_{[i]}^{\T}$. Then $X\in\fraks$ if and only if there exists $i\neq j\leq k$ and $l\in\mathbb{Z}\setminus \{0\}$, such that $\theta_i\pm\theta_j = 2\pi l$, i.e., 
    \begin{equation}\label{eq:tangent-conjugate-locus}
        \fraks= \big\{
                \textstyle \outerprodskew{R}{\theta}: \exists i\neq j\leq k, l\in\mathbb{Z}\setminus\{0\}, \text{ s.t. } \theta_i\pm \theta_j = 2l\pi
        \big\}.
    \end{equation}
\end{proposition}
Note that the introduction of the auxiliary angle makes the description of $\fraks$ in~\eqref{eq:tangent-conjugate-locus} more streamlined than in~\cite[Theorem 3.1]{deng2025expskew}.

Consider $X = \outerprodskew{R}{\alpha}$ with the angles $\alpha\in\realset^k$ under $R$ and let $Q = \exp(X)$. Notice that the Schur basis $R$ of $X$ is also a Schur basis of $Q$ (as $Q = \exp(X) = R\exp(R^{\T}XR)R^{\T}$ is a Schur decomposition). Let $\theta\in \prng^k$ be the vector of principal angles of $Q$ under $R$, it holds that
\begin{equation}\label{eq:shifted-angles-22}
    \exp\left(\sbmatrix{0 & -\alpha_i\\\alpha_i & 0}\right) = \big[\begin{smallmatrix}
                \cos(\theta_i) & -\sin(\theta_i)\\ \sin(\theta_i) & \cos(\theta_i)
            \end{smallmatrix}\big] \Leftrightarrow \exists x_i\in\mathbb{Z}, \alpha_i = \theta_i + 2\pi x_i\; \text{for $i\leq m$.}
\end{equation}
Thus, the conjugate locus $\frakq$, consisting of all $Q$ with a preimage element in $\frak{S}$ (in view of~\cref{def:conjugate-locus}), is characterized by pairs of coincident principal angles with equal magnitudes, as given in~\cref{coro:conjugate-locus}.

\begin{corollary}\label{coro:conjugate-locus}
    Let $\theta \in \prng^k$ (the auxiliary $\theta_k = 0$ when $n = 2k-1$) be the principal angles of $Q\in\so_n$ under a Schur basis $R$, i.e., $Q = R\exp(\Theta)R^{\T}$ where $\Theta$ is block diagonal matrix with $\Theta_{[i,i]} = \big[\begin{smallmatrix}
        0 & -\theta_i\\\theta_i & 0
    \end{smallmatrix}\big]$ for $i \leq m$ and $\Theta_{[k,k]} = \sbmatrix{0}$ when $n = 2k-1$. Then, $Q\in\frakq$ if and only if there exists $i\neq j\leq k$ such that $|\theta_i| = |\theta_j|$, i.e.,
    \begin{equation}\label{eq:conjugate-locus}
        \frakq = \left\{R\exp(\Theta) R^{\T}:R\in\orth_n, \theta\in\prng^k, \exists i\neq j\leq k, |\theta_i| = |\theta_j|\right\}.
    \end{equation}
\end{corollary}
\begin{proof}
    For any $Q = R\exp(\Theta)R^{\T}$ satisfying the constraint in~\eqref{eq:conjugate-locus}, construct a preimage of $Q$ in the form of $X = RAR^{\T}$ where $A$ consists of the diagonal blocks $A_{[i,i]} = \sbmatrix{0 & -\alpha_i\\\alpha_i & 0}$ with $\alpha_i = \theta_i +2\pi x_i$ for some $x_i\in\mathbb{Z}, i \leq m$. If $\theta_i = \theta_j$ let $x_i - x_j = 1$ such that $\alpha_i - \alpha_j = \theta_i - \theta_j +2\pi (x_i - x_j) = 2\pi$, i.e., $X \in \fraks$. Similarly, if $\theta_i = -\theta_j$ let $x_i+x_j = 1$ such that $X\in\fraks$. Therefore, $\frakq\subset  \{R\exp(\Theta) R^{\T}\in\so_n:\theta\in\prng^k, \exists i\neq j\leq k, |\theta_i| = |\theta_j|\}$.
    
    We next show that all $Q$ in $\frakq=\{\exp(X):X\in\fraks\}$ meet the constraint in~\eqref{eq:conjugate-locus}. Consider the angles $\alpha\in\realset^k$ of $X$ and the principal angles $\theta\in\prng^k$ of $Q$ under the same Schur basis $R$. In view of~\eqref{eq:tangent-conjugate-locus}, $X\in\fraks\Leftrightarrow \exists i\neq j\leq k, \alpha_i +\alpha_j = 2\pi l$ or $\alpha_i - \alpha_j = 2\pi l$ for some nonzero $l\in\mathbb{Z}\setminus\{0\}$. In view of~\eqref{eq:shifted-angles-22}, $\alpha_i = \theta_i+2\pi x_i$ holds for some $x_i\in\mathbb{Z}$ and $i\leq m$, while the auxiliary $\alpha_k = \theta_k = 0$ when $n= 2k-1$. Then it further yields that $\alpha_i \pm\alpha_j = 2\pi l \Rightarrow \theta_i\pm \theta_j = 0$, i.e., $\{R\exp(\Theta) R^{\T}\in\so_n:\theta\in\prng^k, \exists i\neq j\leq k, |\theta_i| = |\theta_j|\}\subset \{\exp(X):X\in\fraks\}=\frakq$.
\end{proof}

Consider $Q\in \so_n$ and one of its preimages $X\in\skewm_n$. While it is obvious that a Schur basis $R$ of $X$ is also a Schur basis of $Q$, a Schur basis $V$ of $Q$ is not necessarily a Schur basis of $X$. Instead, $Q$ and its preimage $X$ share all Schur bases if $X$ if $X\in\skewm_n\setminus\fraks$, as reviewed in~\cref{prop:shared-basis}.

\begin{proposition}\cite[Proposition 4.9]{deng2025expskew}\label{prop:shared-basis}
    Consider $Q\in \so_n$ and its preimage $X$ in $\skewm_n\setminus \fraks$, termed a \emph{D-preimage} $X$ for short. If $R$ is a Schur basis of $Q$ in the form of~\eqref{eq:block-diagonal-normal} (i.e., $Q = RER^{\T}$), then $R$ is also a Schur basis of $X$ in the form of~\eqref{eq:block-diagonal-normal} such that $X = \outerprodskew{R}{\alpha}$ with the angles $\alpha\in\realset^k$.
\end{proposition}

\begin{corollary}\label{coro:dpreimage-shifted-angle}
    Let $\theta \in \prng^k$ (the auxiliary $\theta_k = 0$ when $n = 2k-1$) be the principal angles of $Q\in\so_n$ under the Schur basis $R\in\orth_n$. Then, for all D-preimage $X\in\exp^{-1}(Q)\cap(\skewm_n\setminus \fraks)$, there exists an integer vector $\xi = (x_1,\ldots, x_k)\in\mathbb{Z}^k$ (the auxiliary $x_k = 0$ when $n = 2k-1$) such that
    \begin{equation}\label{eq:dpreimage-shifted-angle}
        X = \textstyle\outerprodshift{R}{\theta}{x}.
    \end{equation}
\end{corollary}
\begin{proof}
    In view of \cref{prop:shared-basis}, $X$ can be written as $X = \outerprodskew{R}{\alpha}$, where $\alpha_i = \theta_i + 2\pi x_i$ for $x_i\in\mathbb{Z}$ follows from~\eqref{eq:shifted-angles-22}. 
\end{proof}

Note that the D-preimage constraint is necessary for \cref{coro:dpreimage-shifted-angle} because there exist $X\in\fraks$ and a Schur basis $R$ of $\exp(X) = RER^{\T}$, such that $R$ is not a Schur basis of $X$. For example, consider the matrices
\begin{equation*}
    Q =
    \sbmatrix{
        r & -s & 0 & 0\\
        s & r & 0 & 0\\
        0 & 0 & r & -s\\
        0 & 0 & s & r
    },
    X =
        \sbmatrix{
        0 & -t & 0 & 0\\
        t & 0 & 0 & 0\\
        0 & 0 & 0 & -t - 2\pi\\
        0 & 0 & t + 2\pi & 0
        }, \text{ and } 
    R =
        \sbmatrix{
        0.8 & 0 & -0.6 & 0\\
        0 & 1 & 0 & 0\\
        0.6 & 0 & 0.8 & 0\\
        0 & 0 & 0 & 1
        }.
\end{equation*}
where $r = \cos(t)$ and $s=\sin(t)$. By construction, $X\in \fraks\cap\exp^{-1}(Q)$ and $R$ is a Schur basis of $Q$ as $R^{\T}QR = Q$ is block diagonal. However, because $X$ has different angles $t$ and $t+2\pi$, the same $R$ does not preserve the block diagonal structure; indeed, $R^{\T} X R =
    \sbmatrix{
    0 & -0.8t & 0 & -0.6t -1.2\pi\\
    0.8t & 0 & -0.6t & 0\\
    0 & 0.6t & 0 & -0.8t - 1.6\pi\\
    0.6t + 1.2\pi & 0 & 0.8t + 1.6\pi & 0
    }$. \Cref{prop:preimage-scenario} classifies the possible intersections between the set of preimages $\exp^{-1}(Q)$ and the tangent conjugate locus~$\fraks$. See~\cref{fig:scatter-pattern} and~\cref{remark:scatter-pattern} for an interpretation.

\begin{proposition}\label{prop:preimage-scenario}
    Consider $Q\in\so_n$, the following statements hold:
    \begin{enumerate}[(i)]
        \item $\exp^{-1}(Q)\subset\fraks$ if and only if, for all vectors of principal angles $\theta\in\prng^k$ of $Q$, there exists $i\neq j\leq m$ such that $\theta_i = \theta_j = \pi$ holds\footnote{In view of~\cref{remark:angles_symmetry}, the vectors of principal angles are symmetric with respect to permutations and sign-flips, and $-\pi$ is not allowed. Therefore, if there exists a vector of principal angles $\theta$ of $Q$ satisfying $\theta_i = \theta_j =\pi$ for $i\neq j$, this statement then holds for all vectors of principal angles of $Q$.}.
        \item $\exp^{-1}(Q)\subset(\skewm_n\setminus\fraks)$ if and only if $Q\in\so_n\setminus\frakq$.
        \item Otherwise, $\exp^{-1}(Q)\cap\fraks\neq \emptyset$ and $\exp^{-1}(Q)\cap(\skewm_n\setminus\fraks)\neq \emptyset$.
    \end{enumerate}
\end{proposition}
\begin{proof}

    ``If'' statement in (i):  Consider any preimage of $Q\in\so_n$ in the form of
    \begin{equation*}
        X=\textstyle\outerprodskew{R}{\alpha}=\outerprodshift{R}{\theta}{x}\in\exp^{-1}(Q)
    \end{equation*}
    where $\theta\in\prng^k$ are the principal angles of $Q$ under the Schur basis $R$ (the auxiliary $\theta_k = 0$ and $x_k=0$ when $n = 2k-1$). If $\theta_i = \theta_j = \pi$, then $\alpha_i\pm\alpha_j = 2\pi(1 + x_i\pm x_j)$ for $x_i,x_j\in\mathbb{Z}$, which yields $X\in\mathfrak{S}$, i.e., $\exp^{-1}(Q)\subset \fraks$. 
    
    ``Only-if'' statement in (i): Let $\theta_{r}$ and $\theta_{s}$ attain the largest and the second largest magnitude among all principal angles of $Q$, i.e., $|\theta_r|\geq |\theta_s|\geq |\theta_i|$ for all $i\neq r$ and $i\neq s$. Contradiction is used to show that $|\theta_r|= |\theta_s| = \pi$ when $\exp^{-1}(Q)\subset\fraks$. Consider the preimage $X = R\Theta R^{\T}$ of $Q$. Suppose $|\theta_s| < \pi$, then $X = R\Theta R^{\T}\notin \fraks$, which contradicts with the assumption $\exp^{-1}(Q)\subset \fraks$. Indeed, if $|\theta_s|<\pi$, then for any pair of angles $(\theta_i,\theta_j), i\neq j\leq k$, it holds that
    \begin{equation*}
        |\theta_i\pm\theta_j|<|\theta_i|+|\theta_j|\leq |\theta_r|+|\theta_s|\leq \pi + |\theta_s| < 2\pi
    \end{equation*}
    In view of~\eqref{eq:tangent-conjugate-locus}, no rank-deficient constraint is met for $X = R\Theta R^{\T}$, i.e., $X\notin\fraks$. In conclusion, if $\exp^{-1}(Q) \subset\fraks$, then $|\theta_r| = |\theta_s| = \pi$, which yields $\theta_r = \theta_s = \pi$ because of $\theta\in\prng^k$.
    
    Statement (ii) follows from the definition $\frakq = \{\exp(X):X\in\fraks\}$, which yields $Q\in\frakq\Leftrightarrow\exp^{-1}(Q)\cap \fraks\neq \emptyset$, i.e., $Q\notin\frakq\Leftrightarrow\exp^{-1}(Q)\cap \fraks = \emptyset$. Finally, precluding statements (i) and (ii) yields statement (iii).
\end{proof}

\subsection{Separation of D-Preimage}

Recall that a diffeomorphism $f:\mathcal{X}\to\mathcal{Y}$ is equivalent to $f$ being smoothly bijective in $\mathcal{X}$ and $\der f(x):T_x\mathcal{X}\to T_{f(x)}\mathcal{Y}$ being invertible for all $x\in\mathcal{X}$. With $\dexpof{X}:\skewm_n\to T_{\exp(X)}\so_n$ being invertible for all $X\in\mathcal{C}_e,\forall e\in\mathcal{E}$, it remains to find a neighborhood of some reference point $S\in\mathcal{C}_e$ in the same D-component $\mathcal{C}_e$, in which $\exp$ is bijective. To meet this bijective constraint,~\cite[Theorem 4.11]{deng2025expskew} considers the points that are sufficiently close to the reference point $S$ (within $\pi$ radius under the matrix $2$-norm). A refined analysis presented in this section shows that $\exp$ is diffeomorphic on the entire component $\mathcal{C}_e$, when $\mathcal{C}_e\neq\mathcal{C}_*$. Moreover, all D-preimages of any $Q\in\so_n$ within $\mathcal{C}_*$---of which there are at most two---are completely characterized.

Consider a D-component $\mathcal{C}_e$ and $\exp^{-1}(Q)=\{X\in\skewm_n:\exp(X) = Q\}$ of a given $Q\in\so_n$. Notice that $\exp$ is bijective in $\mathcal{C}_e$ if and only if, for all $Q\in\so_n$, $\exp^{-1}(Q)\cap \mathcal{C}_e$ has at most one element. Thus, we proceed to find the counterexample of $Q$ and $\mathcal{C}_e$ where $\exp^{-1}(Q)\cap\mathcal{C}_e$ has two or more elements.

Let $X$ and $Y$ be two skew-symmetric matrices in $\mathcal{C}_e$ with the same exponential $Q = \exp(X) = \exp(Y)$. Thus, $X$ and $Y$ are D-connected by $S(t):[0,1]\to \mathcal{C}_e$ where $S(0) = X$ and $S(1) = Y$. Let $\alpha(t):[0,1]\to \realset^k$ be the continuous angles of $S(t)$ under the continuous Schur basis\footnote{In view of the continuity of eigendecomposition with respect to matrices (see, e.g., \cite[Theorem 7.2.2]{golub2013matrix}), any continuous curve of matrices admits a continuous choice of eigenvectors in which the eigenvalues are continuous. The continuous angles $\theta(t)$ and the associated continuous Schur basis $R(t)$ of $S(t)$ follow from the continuous eigendecomposition.} $R(t)$ along $t\in [0,1]$, i.e.,
\begin{equation}\label{eq:continuous-skew-curve}
    S(t) = \textstyle\sum_{i=1}^m R(t)_{[i]}\big[ \begin{smallmatrix}0 & -\alpha_i(t)\\\alpha_i(t) & 0\end{smallmatrix}\big] R(t)_{[i]}^{\T}.
\end{equation}

Let $S(t)$ be a D-curve with $\exp(S(0)) = \exp(S(1)) = Q$ and let $\alpha(t):[0,1]\to\realset^k$ be a vector of continuous angles of $S(t)$ under some Schur basis $R(t)$. The sum and difference of a pair of angles $\alpha_i(t)$ and $\alpha_j(t)$ forms a continuous curve $(u_{ij}^{\alpha}(t) = \alpha_i(t)+\alpha_j(t), \alpha_i(t) - \alpha_j(t)):[0,1]\to\realset^2$, for $i\neq j\leq k$. Moreover, $|\alpha_i(t)\pm \alpha_j(t)|\neq 2\pi l$ holds for all $t\in [0,1]$ and $l\in\mathbb{Z}\setminus \{0\}$, as $S(t)\in \skewm_n\setminus\fraks$, and the curve $(u_{ij}^{\alpha}(t),v_{ij}^{\alpha}(t))$ for $t\in [0, 1]$ intersects neither $u = 2\pi l$ nor $v = 2\pi l$ for all $l \in \mathbb{Z}\setminus \{0\}$. This is summarized in~\cref{def:uv-chessboard} and~\cref{lemma:dconnected-uv}.

\begin{definition}\label{def:uv-chessboard}
    The \emph{$(u,v)$-chessboard} refers to $\realset^2$ with the chessboard pattern in the tiles split by $u = 2\pi l$ and $v = 2\pi l$ for any integers $l\in\mathbb{Z}$, see~\cref{fig:chessboard}. Note that when $l\neq 0$, the $u = 2\pi l$ or $v = 2\pi l$ corresponds to a rank-deficient condition in~\eqref{eq:tangent-conjugate-locus}. For the continuous angles $\alpha(t):[0,1]\to\realset^k$, the \emph{angular trajectories of $\alpha(t)$ in the $(u,v)$-chessboard} refers to the $(k^2-k)$ continuous curves 
    \begin{equation*}
        \big(u_{ij}^{\alpha}(t)=\alpha_i(t) + \alpha_j(t), v_{ij}^{\alpha}(t)=\alpha_i(t) - \alpha_j(t)\big): [0,1]\to\realset^2,\; \text{for $i \neq j\leq k$.}
    \end{equation*}
\end{definition}

\begin{lemma}\label{lemma:dconnected-uv}
    A continuous curve $S(t):[0,1]\to\skewm_n$, is a D-curve if and only if all angular trajectories $(u^{\alpha}_{ij}, v^{\alpha}_{ij}), i\neq j\leq k$ of continuous angles $\alpha(t)$ of $S(t)$ in the $(u,v)$-chessboard do not meet the boundaries $u = 2\pi l$ nor $v = 2\pi l$ for $l \in\mathbb{Z}\setminus\{0\}$. 
\end{lemma}
\begin{proof}
    This is a graphic characterization of~\eqref{eq:tangent-conjugate-locus} for $S(t)\in(\skewm_n\setminus\fraks)$ at all $t\in[0,1]$.
\end{proof}

\begin{figure}[tbp]
    \centering
    \begin{subfigure}[b]{0.4\textwidth}
        \centering
        \includegraphics[width=\textwidth, height=\textwidth]{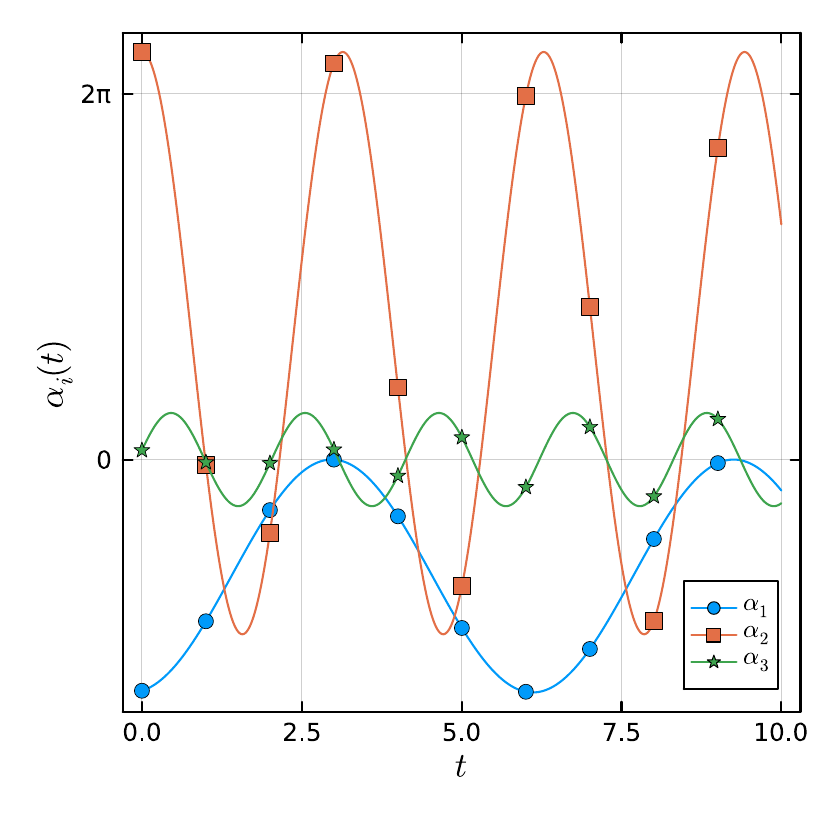}
        \caption{$\alpha_1,\alpha_2,\alpha_3:[0,1]\to \realset$.}
        \label{fig:continuous-angles}
    \end{subfigure}
    \hfill
    \begin{subfigure}[b]{0.4\textwidth}
        \centering
        \includegraphics[width=\textwidth, height=\textwidth]{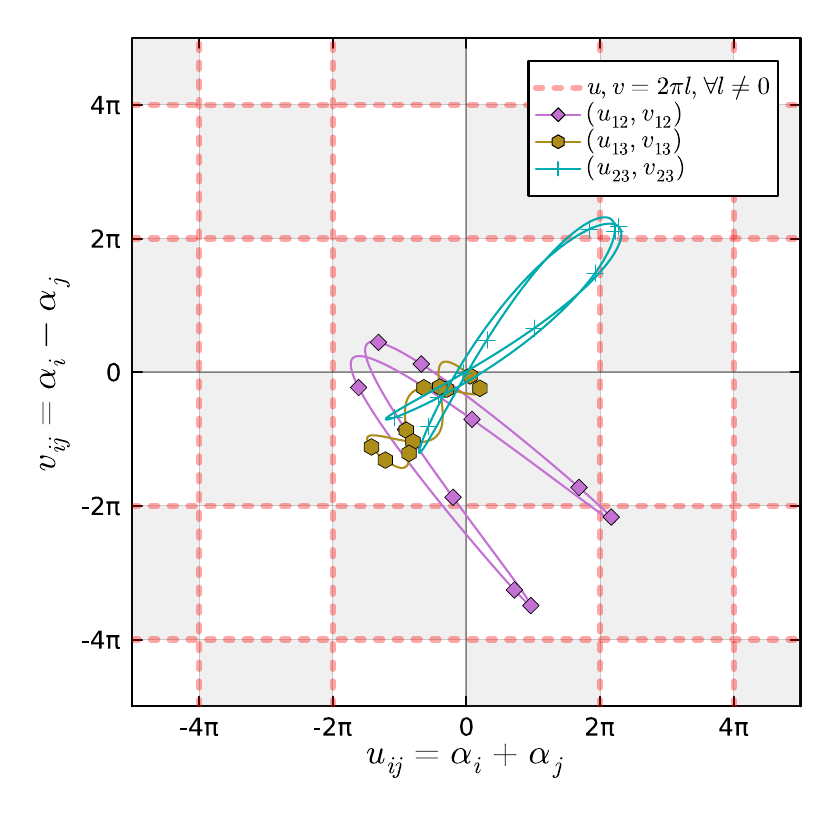}
        \caption{Angular trajectories.}
        \label{fig:chessboard}
    \end{subfigure}
    \caption{Illustration of angular trajectories on the $(u,v)$-chessboard.}\label{fig:angular-trajectories-on-chessboard}
\end{figure}

In addition to the continuous vector of angles obtained in~\eqref{eq:continuous-skew-curve}, we impose the further constraint that each angle is registered with the same principal angle at both $t = 0$ and $t = 1$, as specified in~\eqref{eq:continuous-angles}. \Cref{prop:continuous-angles} shows that this additional constraint can be achieved by relaxing the continuity of $R(t)$.

\begin{proposition}\label{prop:continuous-angles}
    Let $S(t):[0,1]\to \skewm_n\setminus\fraks$ be a continuous curve with $Q = \exp(S(0))=\exp(S(1))$, and let $\theta$ be the vector of principal angles of $Q$ under a Schur basis $R$. Then, $S(t)$ can be written as
    \begin{equation}\label{eq:continuous-angles}
        S(t) = \sum_{i=1}^m \tilde{R}(t)_{[i]}\big[ \begin{smallmatrix}0 & -\tilde{\alpha}_i(t)\\\tilde{\alpha}_i(t) & 0\end{smallmatrix}\big] \tilde{R}(t)_{[i]}^{\T}, \; \text{s.t. }\begin{cases}
            \tilde{\alpha}_i(0) = \theta_i+2\pi x_i& x_i\in\mathbb{Z},\,i\leq m,\\
            \tilde{\alpha}_i(1) = \theta_i+2\pi y_i& y_i\in\mathbb{Z},\,i\leq m,
        \end{cases}
    \end{equation}
    where the vector of angles $\tilde{\alpha}(t):[0,1]\to \realset^k$ is continuous, but the Schur basis $\tilde{R}(t):[0,1]\to \orth_n$ is not required to be continuous.
\end{proposition}
\begin{proof}
    Consider the continuous Schur decomposition of $S(t)$ in the form of~\eqref{eq:continuous-skew-curve} with the continuous angles $\alpha(t)$ and continuous Schur basis $R(t)$. Consider the loop $Q(t) = \exp(S(t)) = R(t)E(t)R(t)^{\T}$ and let $\theta:=\theta(0)$ be the principal angles of $Q$ under $R(0)$, which is solved from $E(0)$. Thus $\alpha(0) = \theta + 2\pi \xi$ follows from~\cref{coro:dpreimage-shifted-angle}.

    The principal angles $\theta(1)$ of $Q$ under $R(1)$ solved from $E(1)$, however, is not guaranteed to be equal to $\theta = \theta(0)$. Rather, it may be the permuted and sign-flipped $\theta = \theta(0)$, i.e., $\theta(1)$ is in the form of $(\delta_1\theta_{p_1}, \ldots, \delta_k\theta_{p_k})$ where $(p_1,\ldots, p_k)$ is a permutation and $\delta_i \in\{\pm 1\}$. Thus, we have $\alpha_i(1) = \delta_i\theta_{p_i} + 2\pi y_i$ for $y_i\in\mathbb{Z}$. 
    
    It suffices to consider the case where $(p_1,\ldots, p_k)$ is a transposition, i.e., only the indices $i$ and $j$ are exchanged for some $i\neq j$. Since all blocks $E_{[i,i]}(t)$ are registered by $R_{[i]}(t)$ continuously along $Q(t)$, this exchange can only occur at some time $t_*\in(0,1)$ where the two blocks become indistinguishable, namely $E_{[i,i]}(t_*) = E_{[j,j]}(t_*)$ when $\delta_i =\delta_j= 1$, or $E_{[i,i]}(t_*) = E_{[j,j]}(t_*)^{\T}$ when $\delta_i=\delta_j = -1$. It follows that $\alpha_i(t_*) = \delta_i\alpha_j(t_*) + 2\pi l$ for some $l\in\mathbb{Z}$. However, $S(t)\in \skewm_n\setminus\fraks$ implies $\alpha_i(t)\pm\alpha_j(t) \neq 2\pi l$ holds for all $t\in [0,1]$ and $l\in\mathbb{Z}\setminus\{0\}$. Thus, we obtain $\alpha_i(t_*)=\delta_i\alpha_j(t_*)$.

    At $t=t_*$, define
    \begin{equation*}
    (\tilde{\alpha}_i(t),\tilde{R}_{[i]}(t))
        =
        \begin{cases}
        (\alpha_i(t), R_{[i]}(t)), & t\leq t_*,\\
        (\alpha_j(t), R_{[j]}(t)), & t>t_*,\ \delta_i=1,\\
        (-\alpha_j(t), [R_{2j}(t)\; R_{2j-1}(t)]), & t>t_*,\ \delta_i=-1,
        \end{cases}
    \end{equation*}
    and
    \begin{equation*}
        (\tilde{\alpha}_j(t),\tilde{R}_{[j]}(t))
        =
        \begin{cases}
        (\alpha_j(t), R_{[j]}(t)), & t\leq t_*,\\
        (\alpha_i(t), R_{[i]}(t)), & t>t_*,\ \delta_i=1,\\
        (-\alpha_i(t), [R_{2i}(t)\; R_{2i-1}(t)]), & t>t_*,\ \delta_i=-1.
        \end{cases}
    \end{equation*}
    Then, by construction, $\tilde{\alpha}_i(t)$ and $\tilde{\alpha}_j(t)$ are continuous and the desired alignments~\eqref{eq:continuous-angles} at the $i$th and $j$th angles are obtained. Note that $\tilde{R}(t)$ is not continuous but
    \begin{equation*}
        \tilde{R}(t)_{[i]}\big[\begin{smallmatrix}
        0 & -\tilde{\alpha}_i(t)\\\tilde{\alpha}_i(t) & 0
    \end{smallmatrix}\big]\tilde{R}(t)_{[i]}^{\T} + \tilde{R}(t)_{[j]}\big[\begin{smallmatrix}
        0 & -\tilde{\alpha}_j(t)\\\tilde{\alpha}_j(t) & 0
    \end{smallmatrix}\big]\tilde{R}(t)_{[j]}^{\T}\; \text{ is continuous.}
    \end{equation*}


    Since every permutation $p$ can be decomposed into finitely many
    transpositions, repeating the above operation finitely many times concludes the proof.
\end{proof}

If the continuous angles $\alpha(t)$ of $S(t)$ satisfies~\eqref{eq:continuous-angles}, then $(u_{ij}^{\alpha}(t), v_{ij}^{\alpha}(t))$ connects $(\theta_i + \theta_j + 2\pi (x_i + x_j), \theta_i + \theta_j + 2\pi (x_i - x_j))$ and $(\theta_i + \theta_j + 2\pi (y_i + y_j), \theta_i + \theta_j + 2\pi (y_i - y_j))$. Among all possible choices of $(x_i, x_j)$ and $(y_i, y_j)$ in $\mathbb{Z}^2$, the two endpoints of $(u_{ij}^{\alpha}(t), v_{ij}^{\alpha}(t))$ admit three possible patterns in~\cref{fig:scatter-pattern}, where every blue cross represents $(\theta_i+\theta_j+2\pi(x_i+x_j), (\theta_i - \theta_j + 2\pi (x_i - x_j)))$ for some $(x_i, x_j)\in\mathbb{Z}^2$. In particular, $|\theta_i| = |\theta_j|=\pi$ yields pattern (a); $|\theta_i|\neq |\theta_j|$ yields pattern (b); and $|\theta_i| = |\theta_j|\neq \pi$ yields pattern (c) with blue crosses on $u = 2\pi l$ or $v = 2\pi l$ for $l \in\mathbb{Z}$. Note that a pair of principal angles with $x_i=x_j=0$ lands in the middle red box.

\begin{figure}[tbp]
    \centering
    \begin{subfigure}[b]{0.32\textwidth}
        \centering
        \includegraphics[width=\textwidth]{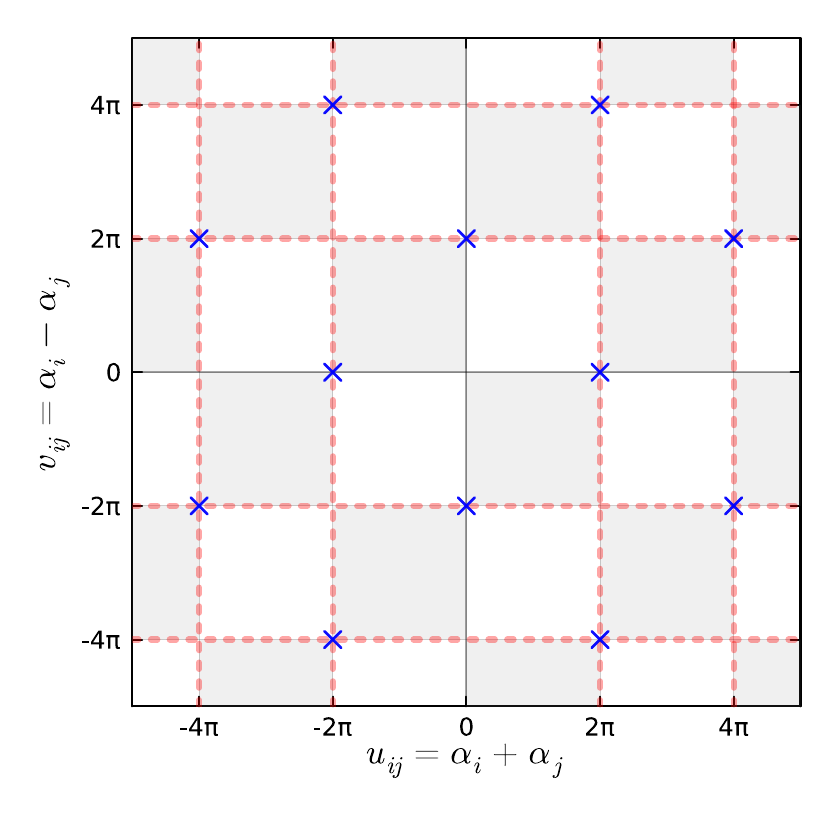}
        \caption*{Pattern (a): $|\theta_i|=|\theta_j|=\pi$.}
    \end{subfigure}
    \hfill
    \begin{subfigure}[b]{0.32\textwidth}
        \centering
        \includegraphics[width=\textwidth]{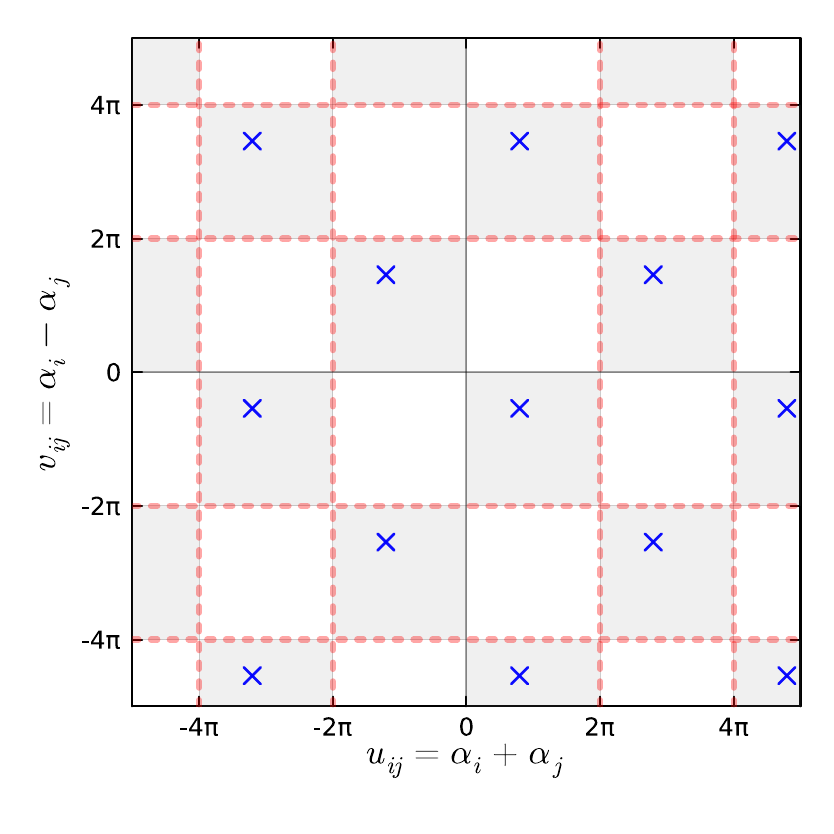}
        \caption*{Pattern (b): $|\theta_i|\neq |\theta_j|$.}
    \end{subfigure}
    \hfill
    \begin{subfigure}[b]{0.32\textwidth}
        \centering
        \includegraphics[width=\textwidth]{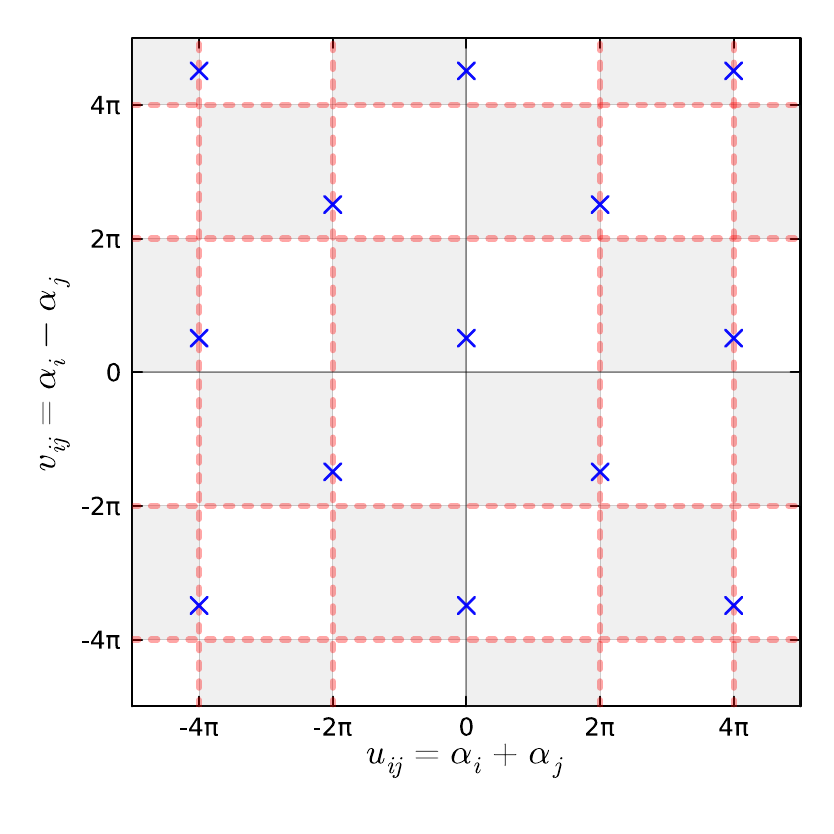}
        \caption*{Pattern (c): $|\theta_i| =|\theta_j|<\pi$.}
    \end{subfigure}
    \caption{Scatter patterns of $(\theta_i + 2\pi x_i, \theta_j + 2\pi x_j)$ in the $(u,\!v)$-chessboard for $(x_i,\!x_j)\!\in\!\mathbb{Z}^2$.\\
    Red lines: Rank-deficient conditions $u = 2\pi l$ and $v = 2\pi l$ for $l\in\mathbb{Z}\setminus\{0\}$.\\
    Blue crosses: the $(u,v)$ representations of $(\theta_i + 2\pi x_i,\theta_j + 2\pi x_j)$ for $(x_i,x_j)\in\mathbb{Z}^2$.}\label{fig:scatter-pattern}
\end{figure}

\begin{remark}\label{remark:scatter-pattern}
    \Cref{fig:scatter-pattern} also provides a graphic interpretation of~\cref{prop:preimage-scenario}. Pattern (a) corresponds to a pair of principal angles satisfying $\theta_i = \theta_j=\pi$. In view of~\cref{prop:preimage-scenario}[(i)], $\exp^{-1}(Q)\subset\fraks$ in this case, which is consistent with all blue crosses located on a red line. Pattern (b) corresponds to a pair of principal angles with distinct magnitudes, i.e., $|\theta_i| \neq |\theta_j|$. In view of~\eqref{eq:conjugate-locus}, if all principal angles of $Q$ have distinct magnitudes, then $Q\notin \frakq$ as described in~\cref{prop:preimage-scenario}[(ii)]. Finally, if no pair of principal angles attains pattern (a) but not all pairs attain pattern (b), then there exists a pair that attains pattern (c) with $|\theta_i| = |\theta_j|<\pi$, as described in~\cref{prop:preimage-scenario}[(iii)].
\end{remark}

\begin{proposition}\label{prop:dconnected-dpreimage}
    Consider $Q\in\so_n$ with the principal angles $\theta\in\prng^k$ (the auxiliary $\theta_k = 0$ when $n = 2k-1$) under some Schur basis $R$. If $X,Y\in\exp^{-1}(Q)$ are D-connected and $X\neq Y$, then the following statements hold:
    \begin{enumerate}[(i)]
        \item $X$ and $Y$ are in the special D-component $\mathcal{C}_*$, i.e., $X, Y\in \mathcal{C}_*$;
        \item There exists a unique $\hat\imath\leq k$ that attains the largest angles: $\hat\imath = {\arg\max}_{i\leq k}|\theta_i|$;
        \item $X$ and $Y$ are the only D-connected preimages of $Q$, and they are given by
        \begin{equation}\label{eq:dconnected-dpreimage}
            \begin{cases}
                X = \textstyle \outerprodskew{R}{\theta},\\
                Y = \textstyle \outerprodskew{R}{\theta} - \mathrm{sign}(\theta_{\hat\imath})R_{[\hat\imath]} \big[\begin{smallmatrix}0 & -2\pi\\2\pi & 0\end{smallmatrix}\big] R_{[\hat\imath]}^{\T}.
            \end{cases}
        \end{equation}
    \end{enumerate}
\end{proposition}
\begin{proof}
    Since $X$ and $Y$ are D-connected, let $S(t):[0,1]\to \mathcal{C}_e$ be the D-curve that connects $S(0) = X$ and $S(1) = Y$ and let $\alpha_i(t):[0,1]\to\realset, i\leq k$ be the continuous angles subject to~\eqref{eq:continuous-angles}. In view of~\eqref{eq:dpreimage-shifted-angle}, for any given Schur basis $R$ of $Q$, the distinct preimages $X\neq Y$ must differ in at least one angle, denoted by the $\hat\imath$th angle. Then, $\alpha_{\hat\imath}(0)\neq \alpha_{\hat\imath}(1)$. Consequently, the angular trajectory of $(u_{\hat\imath j} = \alpha_{\hat\imath}+\alpha_j, v_{\hat\imath j} = \alpha_{\hat\imath}-\alpha_j)$ for all $j\neq \hat\imath$ must connect two distinct blue crosses without meeting a red line. This can only happen in the middle red box $\{(u, v):|u|<2\pi, |v|<2\pi\}$ in~\cref{fig:scatter-pattern}[Pattern (b)]. In particular, $|\alpha_{\hat\imath}(0)\pm\alpha_j(0)|<2\pi$, for all $j\neq \hat\imath$ yields $S(0) = X\in \mathcal{C}_* = \{S\in\skewm_n: |\theta_i\pm\theta_j|<2\pi,\forall i\neq j\leq k\}$. Similarly, $|\alpha_{\hat\imath}(1)\pm\alpha_j(1)|<2\pi$, for all $j\neq \hat\imath$ yields $S(1) = Y\in\mathcal{C}_*$. Thus statement (i) holds.

    Recall that all blue crosses in~\cref{fig:scatter-pattern} are given by $(\theta_i+2\pi x_i, \theta_j + 2\pi x_j)$ for some $(x_i, x_j)\in\mathbb{Z}^2$ and $\theta_i,\theta_j\in \prng$. Therefore, the two blue crosses in the middle red box in~\cref{fig:scatter-pattern}[Pattern (b)] correspond to either $(\theta_i,\theta_j)$ and $(\theta_i-2\pi\mathrm{sign}(\theta_i), \theta_j)$, or $(\theta_i,\theta_j)$ and $(\theta_i, \theta_j-2\pi\mathrm{sign}(\theta_j))$. In other words, one angle is changed by $2\pi$ and the other angle remains the same. In the above derivation, we have concluded that $(u_{\hat\imath j},v_{\hat\imath j})$ connects two blue crosses for all $j\neq \hat\imath$, and $\alpha_{\hat\imath}(0) \neq \alpha_{\hat\imath}(1)$. Thus, $\alpha_j(t)$ remains the same, i.e., $\alpha_j(0) = \alpha_j(1) =\theta_j$, for all $j\neq \hat\imath$. Moreover, since the blue crosses stayed within the middle red box $\{(u, v):|u|<2\pi, |v|<2\pi\}$, it holds that
    \begin{equation*}
        |\theta_{\hat\imath}\pm \theta_j|<2\pi,\qquad |\theta_{\hat\imath}-2\pi \mathrm{sign}(\theta_{\hat\imath}) \pm\theta_j|<2\pi,\qquad \forall j\neq \hat\imath.
    \end{equation*}
    Notice that $|\theta_{\hat\imath}-2\pi \mathrm{sign}(\theta_{\hat\imath})| = 2\pi - |\theta_{\hat\imath}|$ since $\theta_{\hat\imath}\in\prng$. The triangular inequality yields
    \begin{equation*}
        \max\{|\theta_{\hat\imath}-2\pi \mathrm{sign}(\theta_{\hat\imath})+\theta_j|,\ |\theta_{\hat\imath}-2\pi \mathrm{sign}(\theta_{\hat\imath})-\theta_j|\}
        = 2\pi-|\theta_{\hat\imath}|+|\theta_j|<2\pi.
    \end{equation*}
    Hence $|\theta_j|<|\theta_{\hat\imath}|$ holds for all $j\neq {\hat\imath}$, i.e., the statement (ii) follows.
    
    Finally, since $\alpha_j(0) = \alpha_j(1) = \theta_j$ are fixed for all $j\neq \hat\imath$, while $\alpha_{\hat\imath}(0), \alpha_{\hat\imath}(1)$ have only two candidates $\theta_{\hat\imath}$ and $\theta_{\hat\imath}-2\pi \mathrm{sign}(\theta_{\hat\imath})$, the characterizations of $X$ and $Y$ in~\eqref{eq:dconnected-dpreimage} follows.
\end{proof}

\begin{theorem}\label{thm:component-diffeomorphism}
    The exponential map $\exp:\mathcal{C}_e\to\calz_e:=\{\exp(X):X\in\mathcal{C}_e\}$ on a D-component $\mathcal{C}_e$ is a diffeomorphism when $\mathcal{C}_e\neq \mathcal{C}_*$.
\end{theorem}
\begin{proof}
    In view of~\cref{prop:dconnected-dpreimage}, $\exp:\mathcal{C}_e\to\calz_e$ is bijective in $\mathcal{C}_e$ when $\mathcal{C}_e\neq \mathcal{C}_*$. In view of the inverse function theorem, the bijective $\exp:\mathcal{C}_e\to\calz_e$ with invertible $\dexpof{X}$ for all $X\in\mathcal{C}_e$ is a diffeomorphism.
\end{proof}

Compared with the nearby matrix logarithm, see~\cref{def:nearby-logarithm}, restricted on $\{X\in\skewm_n:\|X-S\|_2<\pi\}\cap \mathcal{C}_e$, \cref{thm:component-diffeomorphism} relaxes the distance constraint when $\mathcal{C}_e\neq \mathcal{C}_*$. 

\subsection{Three-Dimensional Rotations: Preimages and Quaternions}\label{subsec:so3}

Consider $\so_3$ which is widely used in various applications. In this case, all skew-symmetric matrices in $\skewm_3$ have only one non-auxiliary angle $\alpha_1\in\realset$ and one auxiliary angle $\alpha_2=0$. For $Q\in\so_3$ in the form of
\begin{equation}\label{eq:so3-q}
    Q = \big[\begin{smallmatrix}
        R_1 & R_2
    \end{smallmatrix}\big]
    \big[\begin{smallmatrix}
        r_1 & -s_1\\
        s_1 & r_1
    \end{smallmatrix}\big] \big[\begin{smallmatrix}
        R_1 & R_2
    \end{smallmatrix}\big]^{\T} + R_3 R_3^{\T},\quad \text{s.t. $r_1^2 + s_1^2 = 1$,}
\end{equation}
the two D-connected preimages of $Q$ in~\eqref{eq:dconnected-dpreimage} simplify to
\begin{equation}\label{eq:so3-double-dpreimages}
    X =
    \big[\begin{smallmatrix}
        R_1 & R_2
    \end{smallmatrix}\big]
    \big[\begin{smallmatrix}
        0 & -\theta_1\\
        \theta_1 & 0
    \end{smallmatrix}\big]
    \big[\begin{smallmatrix}
        R_1 & R_2
    \end{smallmatrix}\big]^{\T},\,
    Y =
    \big[\begin{smallmatrix}
        R_1 & R_2
    \end{smallmatrix}\big]
    \big[\begin{smallmatrix}
        0 & -\theta_1+2\pi\operatorname{sign}(\theta_1)\\
        \theta_1-2\pi\operatorname{sign}(\theta_1) & 0
    \end{smallmatrix}\big]
    \big[\begin{smallmatrix}
        R_1 & R_2
    \end{smallmatrix}\big]^{\T}
\end{equation}
where $\theta_1\in (-\pi,\pi]$ is solved from $\cos(\theta_1) = r_1$ and $\sin(\theta_1) = s_1$.

Quaternions provide another continuous representation of rotations in $\so_3$. Here, a quaternion is an element $q\in\realset^4$ (equipped with the quaternion product). For $Q\in\so_3$ in the form of~\eqref{eq:so3-q}, the standard axis--angle formula for unit quaternions~\cite[Sec.~6.12, Eq.~(175)]{diebel2006representing} yields a representation:
\begin{equation}\label{eq:so3-quaternion}
    q(\alpha_1, R)
    =
    \big[\begin{smallmatrix}
        \cos(\alpha_1/2)\\
        u\sin(\alpha_1/2)
    \end{smallmatrix}\big],\; \text{s.t. $\cos(\alpha_1) = r_1$, $\sin(\alpha_1) = s_1$ and $u = \det(R) R_3$.} 
\end{equation}
In view of~\eqref{eq:dpreimage-shifted-angle}, $\alpha_1 = \theta_1 + 2\pi x_1$ holds for $x_1\in\mathbb{Z}$. Moreover, there are exactly two unit quaternions representing the same $Q$ in~\eqref{eq:so3-q}, namely
\begin{equation}\label{eq:so3-double-quaternion}
    q(\theta_1, R) = \big[\begin{smallmatrix}
        \cos(\theta_1/2)\\
        u\sin(\theta_1/2)
    \end{smallmatrix}\big]\text{ and } -q(\theta_1, R) = \begin{cases}
        q(\theta_1-2\pi \operatorname{sign}(\theta_1), R)& \theta_1\neq 0,\\
        q(2\pi, R) & \theta_1 = 0,
    \end{cases}
\end{equation}
where $\theta_1\in (-\pi, \pi]$. Thus, rotations in $\so_3$ are represented continuously on the unit sphere of quaternions $\{q\in\realset^4: \|q\|_2^2 = 1\}$, where $q$ and $-q$ represent the same rotation matrix. The two quaternions in~\eqref{eq:so3-double-quaternion} with $\alpha_1\in (-2\pi, 2\pi)$ naturally identify the two D-connected preimages in~\eqref{eq:so3-double-dpreimages}.

\begin{theorem}\label{thm:so3-quaternion-identification}
Let $Q\in\so_3\setminus\{I_3\}$ be written in the form of~\eqref{eq:so3-q} with the vector of principal angles $(\theta_1, 0)$ under the Schur basis $R\in\orth_3$. Then, the D-connected preimages $X$ and $Y$ in~\eqref{eq:so3-double-dpreimages} are identified with the two quaternions in~\eqref{eq:so3-double-quaternion} as
\begin{equation}\label{eq:so3-preimage-quaternion}
    X \longleftrightarrow q(\theta_1,R),
    \qquad
    Y \longleftrightarrow -q(\theta_1,R).
\end{equation}
\end{theorem}
\begin{proof}
    Recall that the principal vector $(\theta_1, 0)$ is unique under a given Schur basis $R$. The identification in~\eqref{eq:so3-preimage-quaternion} is follows trivially for the given $R$. It remains to check that this identification is independent of the Schur basis. By the invariance group of Schur bases characterized in~\cite{mataigne2024eigenvalue}, every other Schur basis of $Q$ is of the form $\hat R = \big[\begin{smallmatrix} R_{[1]}O & \delta R_3 \end{smallmatrix}\big]$ where $O\in\orth_2$ and $\delta\in\{-1,1\}$. Under this change of basis, the corresponding vector of principal angles is $(\hat\theta_1,0)=(\det(O)\theta_1,0)$. Since $\det(\hat{R}) = \det(O)\delta \det(R)$, it holds that $q(\hat{\theta}_1, \hat{R}) = \big[\begin{smallmatrix}
        \cos(\det(O)\theta_1/2)\\ \det(\hat{R})\delta R_3\sin(\det(O)\theta_1/2)
    \end{smallmatrix}\big] = \big[\begin{smallmatrix}
        \cos(\theta_1/2)\\ (\delta\det(O))^2\det(R) R_3\sin(\theta_1/2)
    \end{smallmatrix}\big] = q(\theta_1, R)$. Thus, the quaternion $q(\theta_1, R)$ associated with $X$ is independent of the Schur basis, which concludes the proof.
\end{proof}

\section{Conjugate Complement}\label{sec:conjugate-complement}

Having established the distribution of D-preimage(s) in the map $\exp:\mathcal{C}_e\to\calz_e$ for each D-component $\mathcal{C}_e$, this section investigates the structure of the images $\calz_e$ through the conjugate locus $\frakq\subset\so_n$ and its complement. \Cref{prop:closed-conjugate-locus} first shows that $\frakq$ is a closed subset of measure zero in $\so_n$. The canonical Schur decomposition introduced in \cref{def:canonical-decomp} then reveals a distinguished $2\pi$-shift invariant associated with each D-component except for $\mathcal{C}_*$, as established in~\cref{prop:path-connected}. This invariance identifies the path-connected topology of the conjugate complement, ultimately leading to the characterization of all D-components, as presented in~\cref{prop:canonical-shift-dpreimage} and~\cref{thm:canonical-label}.

\subsection{Characterization of the Conjugate Locus}

In view of~\cref{prop:preimage-scenario}[(iii)], there exists $Q\in\frakq$ with D-preimages, i.e., $\exp^{-1}(Q)\cap(\skewm_n\setminus \fraks)\neq \emptyset$. For example,
\begin{equation*}
    Q =
    \sbmatrix{
        c & -s & 0 & 0\\
        s & c & 0 & 0\\
        0 & 0 & c & -s\\
        0 & 0 & s & c
    },
    X =
        \sbmatrix{
        0 & -t & 0 & 0\\
        t & 0 & 0 & 0\\
        0 & 0 & 0 & -t\\
        0 & 0 & t & 0
        }, \text{ where } t\in(-\pi, \pi),\, \begin{cases}
            c = \cos(t),\\
            s = \sin(t),
        \end{cases}
\end{equation*}
satisfies $X \in \exp^{-1}(Q)\cap(\skewm_n\setminus \fraks)$ but $Q \in \frakq$, as characterized in~\eqref{eq:conjugate-locus}. Moreover, when $t = 0$, then $Q = I_4$ is in $\mathfrak{Q}$. It is therefore necessary to distinguish the conjugate complement in $\so_n$ with the tangent conjugate complement in $\skewm_n$ used in~\cref{def:d-notation}.

\begin{definition}\label{def:strict-d}
    All preimages of $Q\in \so_n\setminus \frakq$ are termed \emph{strict D-preimages} of $Q$. Moreover, a \emph{strict D-curve} $S(t):[0,1]\to\skewm_n\setminus\fraks$ is a D-curve consisting of only strict D-preimages. Consequently, $\exp(S(t))$ of a strict D-curve $S(t)$ does not intersect $\frakq$, i.e., $\exp(S(t)):[0,1]\to\so_n\setminus \frakq$. 
\end{definition}

\begin{proposition}\label{prop:graphic-strict-dconnected}
    A continuous $S(t):[0,1]\to\skewm_n$ with a set of continuous angles $\alpha(t):[0,1]\to\realset^k$ in the form of~\eqref{eq:continuous-skew-curve} is a strict D-curve if and only if all angular trajectories $(u^{\alpha}_{ij}(t) = \alpha_i(t) + \alpha_j(t), v^{\alpha}_{ij}(t) = \alpha_i(t)-\alpha_j(t)), i\neq j\leq k$ are strictly contained\footnote{A point $(u,v)$ is strictly contained in a tile if and only if $u \neq 2\pi l$ and $v \neq 2\pi l$ for all $l\in\mathbb{Z}$}. in a single tile of the $(u,v)$-chessboard.
\end{proposition}
\begin{proof}
    In view of~\eqref{eq:shifted-angles-22}, if $u^{\alpha}_{ij}(t) = 2\pi l$ or $v^{\alpha}_{ij} = 2\pi l$ holds for any $l\in\mathbb{Z}$ (including $0$), then the corresponding pair of principal angles of $\exp(S(t))$ satisfies $|\theta_i(t)| = |\theta_j(t)|$, i.e., $\exp(S(t))\in\frakq$. Therefore, all angular trajectories of a strict D-curve do not intersect $u = 2\pi l$ nor $v = 2\pi l$ for all $l\in\mathbb{Z}$. Since the trajectories are continuous, each of them is strictly contained in one tile.
\end{proof}

\Cref{prop:closed-conjugate-locus} shows that the conjugate locus $\frakq$ has measure zero in $\so_n$, ensuring that the subsequent analysis on the complement $\so_n \setminus \frakq$ applies almost everywhere.

\begin{proposition}\label{prop:closed-conjugate-locus}
    $\frakq$ is a closed subset in $\so_n$ with zero measure.
\end{proposition}
\begin{proof}
    In view of the continuity of eigendecomposition (see, e.g.,~\cite[Theorem 7.2.2]{golub2013matrix}), any continuous curve of special orthogonal matrices admits a continuous choice of Schur bases in which the diagonal blocks are also continuous, i.e., 
    \begin{equation*}
        Q(t) = \textstyle\sum_{i=1}^k R(t)_{[i]}E_{[i,i]}(t) R(t)_{[i]}^{\T}
    \end{equation*}
    where $R(t)\in\orth_n$ and $E_{[i,i]}(t) = \sbmatrix{c_i(t) & -s_i(t)\\ s_i(t) & c_i(t)}\in\so_2, i\leq m$ ($E_{[k,k]}(t) \equiv 1$ when $n = 2k-1$) are continuous. Moreover, the magnitude of the $i$-th non-auxiliary principal angles under $R(t)$ is given by $|\theta_i(t)| = \arccos(c_i(t))$, which is continuous. Note that the principal angle itself may be discontinuous when it jumps around $-\pi$ and $\pi$ but the magnitude remains continuous. Denote the minimal gap between principal angles of $Q\in\so_n$ by
    \begin{equation*}
        \mathrm{gap}(Q) := \argmin_{i\neq j\leq k}||\theta_i| - |\theta_j||,
    \end{equation*}
    then $\mathrm{gap}(Q(t))$ is continuous. Moreover, the conjugate locus $\frakq$ in~\eqref{eq:conjugate-locus} is equivalent to $\frakq = \{Q\in\so_n:\mathrm{gap}(Q) = 0\}$, which is then closed in $\so_n$.

    The measure of $\frakq$ follows from the dimensionality argument. Consider the characteristic polynomial $p_Q$ of $Q$ defined as
    \begin{equation*}
        p_Q(\lambda)=\mathrm{det}(\lambda I_n - Q) = \lambda^n + c_{n-1}(A)\lambda^{n-1} + \cdots + c_0(A),\, \text{for $\lambda\in \mathbb{C}$}
    \end{equation*}
    where the eigenvalues of $Q$ are the roots of $p_Q$.  Moreover, consider the discriminant $\Delta_Q$ of $Q$ with the eigenvalues $\lambda_i,i= 1,\ldots,n$ defined as
    \begin{equation*}
        \Delta(Q) = \textstyle \prod_{1\leq i < j \leq n}(\lambda_i-\lambda_j)^2.
    \end{equation*}
    In view of~\eqref{eq:conjugate-locus}, if $Q\in\frakq$, then $Q$ has repeated eigenvalues and $\Delta(Q) = 0$, i.e., $\frakq\subset \{Q\in\so_n:\Delta(Q) = 0\}$. We proceed to show that $\{Q\in\so_n:\Delta(Q) = 0\}$ has a dimension smaller than $\so_n$ by the fact that $\Delta(Q)$ is a polynomial of entries in $Q$. Since the discriminant is invariant under any permutation of the eigenvalues ${\lambda_i}$, it is a symmetric polynomial $\Delta(Q) = \Delta(\lambda_1,\ldots, \lambda_n)$ in the roots of $p_Q$. By the fundamental theorem of symmetric polynomials (see, e.g., \cite[Ch. VI, §2]{lang2012algebra}), the discriminant $\Delta(\lambda_1,\ldots, \lambda_n)$ of $p_Q(\lambda)$ is a polynomial in the coefficients $c_i(Q)$ of $p_Q(\lambda)$. Moreover, since the coefficients of $p_Q$ are polynomial functions of the matrix entries (see, e.g., \cite[§1.2.1]{horn2012matrix}), the discriminant is a polynomial function of the entries of $Q$. Consequently, the dimension of $\{Q\in\so_n:\Delta(Q) = 0\}$ is strictly less than the dimension of $\so_n$. Therefore, $\frakq\subset \{Q\in\so_n:\Delta(p_Q) = 0\}$ has zero measure in $\so_n$.
\end{proof}

\subsection{Canonical Schur Decomposition}

The angular relationship between principal angles of $Q$ and the angles of its D-preimages revealed in~\eqref{eq:dpreimage-shifted-angle} suggests that every D-preimage can be identified by the shifts of some multiples of $2\pi$ to the the principal angles of $Q$. However, the choices of Schur basis allow the principal angles to be arbitrarily permuted and sign-flipped. This subsection introduces a canonical alignment rule in~\cref{def:canonical-decomp} that generates a unique vector of principal angles in any $Q$.

\begin{definition}\label{def:canonical-decomp}
    A Schur decomposition $Q = R\exp(\Theta)R^{\T}\in\so_n$ and the associated principal angles $\theta\in \prng^k$ are termed \emph{canonical} if (i) $R\in\so_n$, and (ii):
    \begin{equation}\label{eq:canonical-angles}
        \begin{cases}
            \pi\geq \theta_1 \geq \cdots \geq \theta_{k-1}\geq |\theta_k| \geq 0, &\text{ if } n = 2k,\\ 
            \pi\geq\theta_1 \geq \cdots \geq \theta_{k-1}\geq |\theta_k| = 0, &\text{ if } n = 2k-1.\\ 
        \end{cases}
    \end{equation}
    Moreover, a Schur decomposition $X =\outerprodskew{R}{\alpha}\in\skewm_n$, the associated angles $\alpha = \theta + 2\pi\xi\in\realset^k$ and angular shift $\xi\in\mathbb{Z}^k$ are termed \emph{canonical} if (i) $R\in\so_n$, and (ii) $\theta\in\prng^k$ satisfies~\eqref{eq:canonical-angles}.
\end{definition}

\begin{remark}
    It is also possible to remove the absolute values for $\theta_k$ when $n = 2k$, but then the condition on $R$ has to be relaxed to $\orth_n$, and this is not suitable for results like~\cref{prop:path-connected} because $\orth_n$ has two connected components.
\end{remark}

Such a canonical Schur decomposition is essential in understanding $\so_n$, as it identifies a unique set of principal angles of any $Q\in\so_n$ as constructed in~\cref{prop:existence-canonical-decomp}, regardless of repeated eigenvalues or principal angles.

\begin{proposition}\label{prop:existence-canonical-decomp}
    Let $Q =R\exp(\Theta)R^{\T}\in\so_n$ be a Schur decomposition with the principal angles $\theta \in \prng^k$ (the auxiliary $\theta_k=0$ when $n = 2k-1$). A canonical Schur decomposition of $Q=V\exp(\Sigma) V^{\T}$ with $V\in\so_n$ and $\sigma$ subject to~\eqref{eq:canonical-angles} are obtained using the following procedure:
    \begin{description}
        \item[Initial.]
        Set $V=R$ and $\sigma=\theta$.
        If $\det(R)=-1$, swap the two columns in the first block
        $V_{[1]}=\sbmatrix{V_1 & V_2}$ and replace $\sigma_1$ by $-\sigma_1$, such that $\det(V) = 1$.

        \item[Shuffle.]
        Let $\rho:\{1,\ldots,k\}\to\{1,\ldots,k\}$ (with $\rho(k)=k$ when
        $n=2k-1$) be a permutation such that
        $|\sigma_{\rho(i)}|\ge |\sigma_{\rho(i+1)}|$ for $i=1,\ldots,k-1$.
        Reorder the angles and Schur blocks accordingly:
        \[
        \sigma\leftarrow(\sigma_{\rho(1)},\ldots,\sigma_{\rho(k)}),\qquad
        V\leftarrow\bigl[\,V_{[\rho(1)]}\ \cdots\ V_{[\rho(k)]}\,\bigr].
        \]

        \item[Flip.]
        For $i=1,\ldots,m-1$, sequentially check whether $\sigma_i<0$.
        If so, swap the two columns within each of the adjacent blocks
        $V_{[i]}$ and $V_{[i+1]}$, and replace
        $(\sigma_i,\sigma_{i+1})$ by $(-\sigma_i,-\sigma_{i+1})$ so that
        $-\sigma_i>0$. The value $\sigma_{i+1}$ is processed at the
        next step if necessary.

        \item[Extra Flip.]
        If $n=2m+1$ and $\sigma_m<0$, swap the two columns in $V_{[m]}$,
        replace $\sigma_m$ by $-\sigma_m>0$, and replace
        $V_n$ by $-V_n$. If $n=2m$ and $\sigma_m = -\pi$, replace $\sigma_m$ by $-\sigma_m=\pi$.
    \end{description}
\end{proposition}
\begin{proof}
    Notice that interchanging columns within $V_{[i]}, i\leq m$ satisfies
    \begin{equation*}
        \begin{aligned}
            V_{[i]}\exp\left(\Sigma_{[i,i]}\right)V_{[i]}^{\T} &= \sbmatrix{V_{2i-1} & V_{2i}} \sbmatrix{\cos(\sigma_i) & -\sin(\sigma_i)\\\sin(\sigma_i) & \cos(\sigma_i)}\sbmatrix{V_{2i-1}^{\T} \\ V_{2i}^{\T}}\\
            &= \sbmatrix{V_{2i} & V_{2i-1}} \sbmatrix{\cos(-\sigma_i) & -\sin(-\sigma_i)\\\sin(-\sigma_i) & \cos(-\sigma_i)}\sbmatrix{V_{2i}^{\T} \\ V_{2i-1}^{\T}}
        \end{aligned}
    \end{equation*}
    and $V_{[k]}\exp(\Sigma_{[k,k]})V_{[k]}^{\T} = V_nV_n^{\T} = (-V_n)(-V_n)^{\T}$ holds when $n = 2k-1$. When $n = 2k$ and $\sigma_m = -\pi$, $V_{[m]}\exp\left(\sbmatrix{0 & \pi\\-\pi & 0}\right)V_{[k]}^{\T} = V_{[m]}\exp\left(\sbmatrix{0 & -\pi\\\pi & 0}\right)V_{[k]}^{\T}$. In conclusion, the outer product terms $R_{[i]}\exp(\Theta_{[i,i]})R_{[i]}^{\T}$ in $Q= R\exp(\Theta)R^{\T}$ are preserved in $V_{[i]}\exp(\Sigma_{[i,i]})V_{[i]}^{\T}$ for $i\leq k$ in the procedure. Meanwhile, the (Shuffle) stage only shuffles the summation as $\sum_{i=1}^k V_{[\rho(i)]} \exp(\Sigma_{[\rho(i),\rho(i)]}) V_{[\rho(i)]}^{\T}$. Then, $Q = R\exp(\Theta)R^{\T} = V\exp(\Sigma)V^{\T}$.

    By construction, the inequalities~\eqref{eq:canonical-angles} are satisfied in $\sigma\in\prng^k$. It remains to show that $V\in\so_n$. Notice that interchanging any pair of columns or flipping any single column in $V$ flips the determinant of $V$. Therefore, $\det(V)$ is initialized with $\det(V) = 1$, and then it is flipped even times in the (Shuffle), (Flip) and the optional (Extra Flip) stages, such that the resulting $\det(V) = 1$ holds, i.e., $V$ on output is in $\so_n$.
\end{proof}

Since any $Q\in\so_n$ admits a canonical Schur decomposition, the conjugate locus can be characterized by the canonical principal angles~\eqref{eq:canonical-angles} without the coincident angles.

\begin{proposition}\label{prop:non-conjugte-canonical-decomp}
    Let $\theta \in \prng^k$ be the canonical principal angles of $Q\in\so_n$, i.e., $\theta$ satisfies~\eqref{eq:canonical-angles}. Then, $Q\notin \frakq$ if and only if $\theta$ satisfies
    \begin{equation}\label{eq:non-conjugate-principal-angles}
        \begin{cases}
            \pi\geq \theta_1 > \cdots > \theta_{k-1}> |\theta_k| \geq 0, &\text{ if } n = 2k,\\ 
            \pi\geq \theta_1 > \cdots > \theta_{k-1}> |\theta_k| = 0, &\text{ if } n = 2k-1.\\ 
        \end{cases}
    \end{equation}
\end{proposition}
\begin{proof}
    $Q\notin \frakq$ if and only if $|\theta_i|\neq |\theta_j|$ holds for all $i\neq j\leq k$, see~\eqref{eq:conjugate-locus}. The canonical principal angles $\theta$ subject to~\eqref{eq:canonical-angles} with distinct magnitudes yields~\eqref{eq:non-conjugate-principal-angles}.
\end{proof}

\subsection{Canonical Shift Label of D-components}

The canonical Schur decomposition ensures that every $Q\in\so_n$ admits a unique vector of canonical principal angles, as established in \cref{prop:existence-canonical-decomp}. This canonical form exposes the unique associated integer shifts $\xi\in\mathbb{Z}^k$ of $2\pi$ of the angles in its D-preimages, which provide the key invariants underlying the D-component structure.

Recall that a continuous $S(t):[0,1]\to \skewm_n$ admits the form of~\eqref{eq:continuous-skew-curve}:
\begin{equation*}
    S(t) = \textstyle\sum_{i=1}^m R(t)_{[i]}\big[\begin{smallmatrix}0 & -\theta_i(t)\\\theta_i(t) & 0\end{smallmatrix}\big] R(t)_{[i]}^{\T}
\end{equation*}
where $\theta(t) = (\theta_1(t),\ldots, \theta_k(t))$ is continuous in $\realset^k$ and $R(t)$ is continuous in $\orth_n$. \Cref{prop:path-connected} shows that for all pairs of endpoints $P$ and $Q$ in $\so_n\setminus \frakq$ and all angular shifts $\xi\in\mathbb{Z}^k$, one can construct a strict D-curve $S(t)$ in the form of~\eqref{eq:strict-dcurve} so that $\exp(S(t))$ connects $P$ and $Q$.

\begin{proposition}\label{prop:path-connected}
    For $P$ and $Q$ in $\so_n\setminus\frakq$ and an angular shift $\xi=(x_1,\ldots, x_k)\in\realset^k$ (the auxiliary $x_k = 0$ when $n = 2k-1$), let $P = V\exp(\Sigma)V^{\T}$ and $Q = W\exp(\Theta) W^{\T}$ be canonical Schur decompositions, see~\cref{def:canonical-decomp}, such that $V, W\in\so_n$ and the respective canonical principal angles $\sigma$ and $\theta$ are subject to~\eqref{eq:non-conjugate-principal-angles}. Consider a continuous curve in $\skewm_n$:
    \begin{equation}\label{eq:strict-dcurve}
        S_{R,\xi}(t) := \textstyle\sum_{i=1}^m R(t)_{[i]}\big[\begin{smallmatrix}0 & -(1-t)\,\sigma_i - t\,\theta_i - 2\pi x_i\\(1-t)\,\sigma_i + t\,\theta_i + 2\pi x_i & 0\end{smallmatrix}\big] R(t)_{[i]}^{\T},
    \end{equation}
    where $R(t)$ is continuous in $\so_n$ and it connects $R(0) = V$ and $R(1) = W$. Then $S_{\xi, R}(t):[0,1]\to\skewm_n$ is a strict D-curve, and its exponential connects $\exp(S_{R,\xi}(0)) = P$ and $\exp(S_{R,\xi}(1))=Q$.
\end{proposition}
\begin{proof}
    The equalities $\exp(S_{R,\xi}(0)) = P$ and $\exp(S_{R,\xi}(1))=Q$ follows by construction. It suffices to show that $\exp(S(t))\notin \frakq$ for all $t\in[0,1]$. Notice that the principal angles of $\exp(S(t))$ are given by the linear interpolant $(1-t)\,\sigma + t\,\theta\in\prng^k$ for $t\in[0,1]$. Moreover, the inequalities~\eqref{eq:non-conjugate-principal-angles} that hold for $\sigma$ and $\theta$ also hold for $(1-t)\,\sigma + t\,\theta$ for all $t\in[0,1]$, which implies that $|(1-t)\,\sigma_i+t\,\theta_i|\neq |(1-t)\,\sigma_j+t\,\theta_j|$ for all $t\in[0,1]$. In view of~\eqref{eq:conjugate-locus}, $\exp(S(t))\notin\frakq$ follows for $t\in[0,1]$.
\end{proof}

Observe that the canonical shift $\xi$ along the strict D-curve $S_{\xi, R}(t)$ remains constant and the points on $S_{\xi, R}(t)$ are naturally D-connected. It suggests that the canonical shift is invariant within a D-component. To see that, we construct the set $\mathcal{C}_{\xi}$ that collects all skew-symmetric matrices in $\skewm_n\setminus\fraks$ with the same canonical shift $\xi$ in~\cref{prop:canonical-shift-dpreimage}.

\begin{proposition}\label{prop:canonical-shift-dpreimage}
    Consider a canonical shift $\xi=(x_1,\ldots, x_k)\in\mathbb{Z}^k$ (the auxiliary $x_k=0$ when $n=2k-1$), and let $\mathcal{C}_{\xi}$ collect all skew-symmetric matrices in $\skewm_n\setminus\fraks$ with their canonical shifts being $\xi$. Then, $\mathcal{C}_{\xi}$ is given by
    \begin{equation}\label{eq:canonical-shift-dpreimage}
        \mathcal{C}_{\xi} := \left\{\outerprodskew{R}{\alpha}:\begin{aligned}
            &R\in\so_n, \alpha = \theta + 2\pi\xi, \theta\text{ s.t.~\eqref{eq:canonical-angles} and } |\theta_2| < \pi\\
            &\text{If $x_i\neq \pm x_{i+1}$, then $\theta_i\neq \pm\theta_{i+1}$ for $i\leq k-1$}
        \end{aligned}\right\}.
    \end{equation}
\end{proposition}
\begin{proof}
    By~\cref{def:canonical-decomp}, the matrices with canonical shift $\xi$ are $X = \sum_{i=1}^m R_{[i]}\sbmatrix{0 & -\alpha_i\\\alpha_i & 0}R_{[i]}^{\T}$ where $R\in\so_n$, $\alpha = \theta + 2\pi\xi$ and $\theta$ satisfies~\eqref{eq:canonical-angles}. It remains to exclude the matrices in $\fraks$. In view of~\cref{prop:preimage-scenario}[(i), (iii)], there are only two scenarios in which $X\in\fraks$: (i) there are two or more principal angles that attain the magnitude of $\pi$, and/or (iii) if there exists $i\neq j\leq k$ satisfying $\theta_i = \pm\theta_j$ but $x_i\neq \pm x_j$ so that $\alpha_i\mp\alpha_j = (\theta_i+2\pi x_i)\mp(\theta_j+2\pi x_j) = 2\pi(x_i\mp x_j)\neq 0$. Since $\theta$ subject to~\eqref{eq:canonical-angles} are ordered by decreasing magnitude, (i) only occurs with $|\theta_2| = \pi$ and (iii) only occurs with adjacent indices $i$ and $j = i+1$ for $i\leq k-1$. Excluding the scenarios (i) and (iii) yields~\eqref{eq:canonical-shift-dpreimage}. 
\end{proof}

\begin{corollary}\label{coro:dcomponent-closure}
    The closure of $\mathcal{C}_{\xi}$ constructed in~\eqref{eq:canonical-shift-dpreimage} is given by 
    \begin{equation*}
        \mathrm{closure}(\mathcal{C}_{\xi})= \left\{\textstyle \outerprodskew{R}{\alpha}:\text{$R\in\so_n, \alpha = \theta+2\pi\xi, \theta$ satisfies~\eqref{eq:canonical-angles}}\right\},
    \end{equation*}
    which has an exponential image $\so_n$, i.e., $\{\exp(X):X\in\mathrm{closure}(\mathcal{C}_{\xi})\} = \so_n$.
\end{corollary}
\begin{proof}
    Denote the set of constrained angles as 
    \begin{equation*}
        \mathcal{A}_{\xi}:= \left\{\theta + 2\pi\xi:\begin{aligned}
            &\theta\text{ satisfies~\eqref{eq:canonical-angles} and } |\theta_2| < \pi\\
            &\text{If $x_i\neq \pm x_{i+1}$, then $\theta_i\neq \pm\theta_{i+1}$ for $i\leq k-1$}
        \end{aligned}\right\},
    \end{equation*}
    and consider the continuous function
    \begin{equation*}
        F:\so_n\times \realset^k \to\skewm_n, (R,\alpha)\mapsto \textstyle\outerprodskew{R}{\alpha}.
    \end{equation*}
    Then, $\mathcal{C}_{\xi} = \{F(R,\alpha):R\in\so_n, \alpha\in\mathcal{A}_{\xi}\}$ and its closure is given by $\mathrm{closure}(\mathcal{C}_{\xi}) = \{F(R,\alpha):R\in\so_n, \alpha\in\mathrm{closure}(\mathcal{A}_{\xi})\}$, where $\mathrm{closure}(\mathcal{A}_{\xi}) = \{\theta+2\pi \xi:\text{$\theta$ satisfies~\eqref{eq:canonical-angles}}\}$.
    In view of~\cref{prop:existence-canonical-decomp}, any $Q\in\so_n$ admits a canonical Schur decomposition $Q= R\exp(\Theta)R^{\T}$ where $R\in\so_n$ and $\theta$ satisfies~\eqref{eq:canonical-angles}. Therefore, $RAR^{\T}$ with $\alpha = \theta + 2\pi\xi$ is in $\mathrm{closure}(\mathcal{C}_{\xi})$ for any $\xi\in\mathbb{Z}^k$. In other words, $\so_n\subset \{\exp(X):X\in\mathrm{closure}(\mathcal{C}_{\xi})\}$ holds for all $\xi\in\mathbb{Z}^k$. On the other hand, $\{\exp(X):X\in\mathrm{closure}(\mathcal{C}_{\xi})\}\subset \{\exp(X):X\in\skewm_n\} = \so_n$, which concludes the proof that $\so_n= \{\exp(X):X\in\mathrm{closure}(\mathcal{C}_{\xi})\}$ for all $\xi\in\mathbb{Z}^k$.
\end{proof}

Note that all points in $\{S_{R,\xi}(t): t\in [0,1]\}\subset \mathcal{C}_{\xi}$ are D-connected, i.e., the curve is in a D-component $\mathcal{C}_e$. Moreover, all $P$ and $Q$ in $\so_n\setminus \frakq$ can be connected by such a strict D-curve with the same canonical shift, the observation is then generalized to $\mathcal{C}_{\xi}\subset \mathcal{C}_e$ as presented in~\cref{prop:component-inclusion}.

\begin{proposition}\label{prop:component-inclusion}
    Consider a canonical shift $\xi=(x_1,\ldots, x_k)\in\mathbb{Z}^k$ (the auxiliary $x_k = 0$ when $n=2k-1$) and the set $\mathcal{C}_{\xi}$ constructed in~\eqref{eq:canonical-shift-dpreimage}. There exists a D-component $\mathcal{C}_e$ such that $\mathcal{C}_{\xi}\subset \mathcal{C}_e$.  In particular, $\mathcal{C}_{(0,0\ldots,0)}\cup \mathcal{C}_{(-1,0,\ldots,0)}\subset\mathcal{C}_*$.
\end{proposition}
\begin{proof}
    It suffices to show that $\mathcal{C}_{\xi}$ are D-connected. Firstly, all strict D-preimages with the canonical shifts $\xi$ are D-connected via the strict D-curve~\eqref{eq:strict-dcurve}. To show that any non-strict D-preimages $X\in \mathcal{C}_{\xi}$ are D-connected with the strict D-preimages in $\mathcal{C}_{\xi}$, let $X$ be located in a D-component $\mathcal{C}_e$. Since $\mathcal{C}_e$ is an open set, see~\cite[Corollary 4.4]{deng2025expskew}, there exists an open ball contained in $\mathcal{C}_e$ that is centered at $X$ with radius $\delta>0$:
    \begin{equation*}
        \mathcal{B}_{X}(\delta):=\{Y\in\skewm_n:\|Y-X\|_2 < \delta\}\subset \mathcal{C}_e.
    \end{equation*}
    
    Let $X = \outerprodskew{V}{\alpha}$ be a canonical Schur decomposition, i.e., $V\in\so_n$, $\alpha=\theta+2\pi\xi$ and $\theta$ satisfies~\eqref{eq:canonical-angles}. Choose $\varepsilon := \{\varepsilon_1,\ldots, \varepsilon_k\}\in (-\delta, \delta)^k$ such that $\theta+t\varepsilon \in \prng^k$ satisfies~\eqref{eq:non-conjugate-principal-angles} for all $t\in (0, 1]$, see an illustration in~\cref{fig:proof-of-shifts}. Consider the curve 
    \begin{equation*}
        X_{\varepsilon}(t):= \textstyle\sum_{i=1}^m V_{[i]}\big[\begin{smallmatrix}0 & -\alpha_i-t\,\varepsilon_i\\ \alpha_i + t\,\varepsilon_i & 0\end{smallmatrix}\big]V_{[i]}^{\T}\qquad \text{where $X_{\varepsilon}(t)\in \mathcal{B}_X(\delta)$ holds for all $t\in [0, 1]$,}
    \end{equation*}
    as $\|X_{\varepsilon}(t)-X\|_2  = \max_{i\leq k}|t\varepsilon_i| < t\delta$. By~\cref{def:canonical-decomp}, the canonical angles of $X_{\varepsilon}(t)$ are given by $\theta+ t\varepsilon+2\pi \xi$ and the canonical shifts remains $\xi$ for all $t\in[0,1]$. Moreover, $X_{\varepsilon}(t), t\in(0, 1]$ are strict D-preimages because the angles $|\theta_i+t\,\varepsilon_i|$ are distinct. In view of~\cref{prop:path-connected}, any strict D-preimages $Y$ with the canonical shift $\xi$ is D-connected with the strict D-preimages $X_{\varepsilon}(1)$ via~\eqref{eq:strict-dcurve}, see illustration in~\cref{fig:proof-of-shifts}. Thus all strict D-preimages $\mathcal{C}_{\xi}$ are D-connected with all non-strict D-preimages in $\mathcal{C}_{\xi}$, which concludes that all points in $\mathcal{C}_{\xi}$ are D-connected. Consequently, $\mathcal{C}_{\xi}$ must be contained in one of the D-components.
    \begin{figure}[tbp]
    \centering
    \begin{subfigure}[b]{0.48\textwidth}
        \centering
        \includegraphics[width=\textwidth]{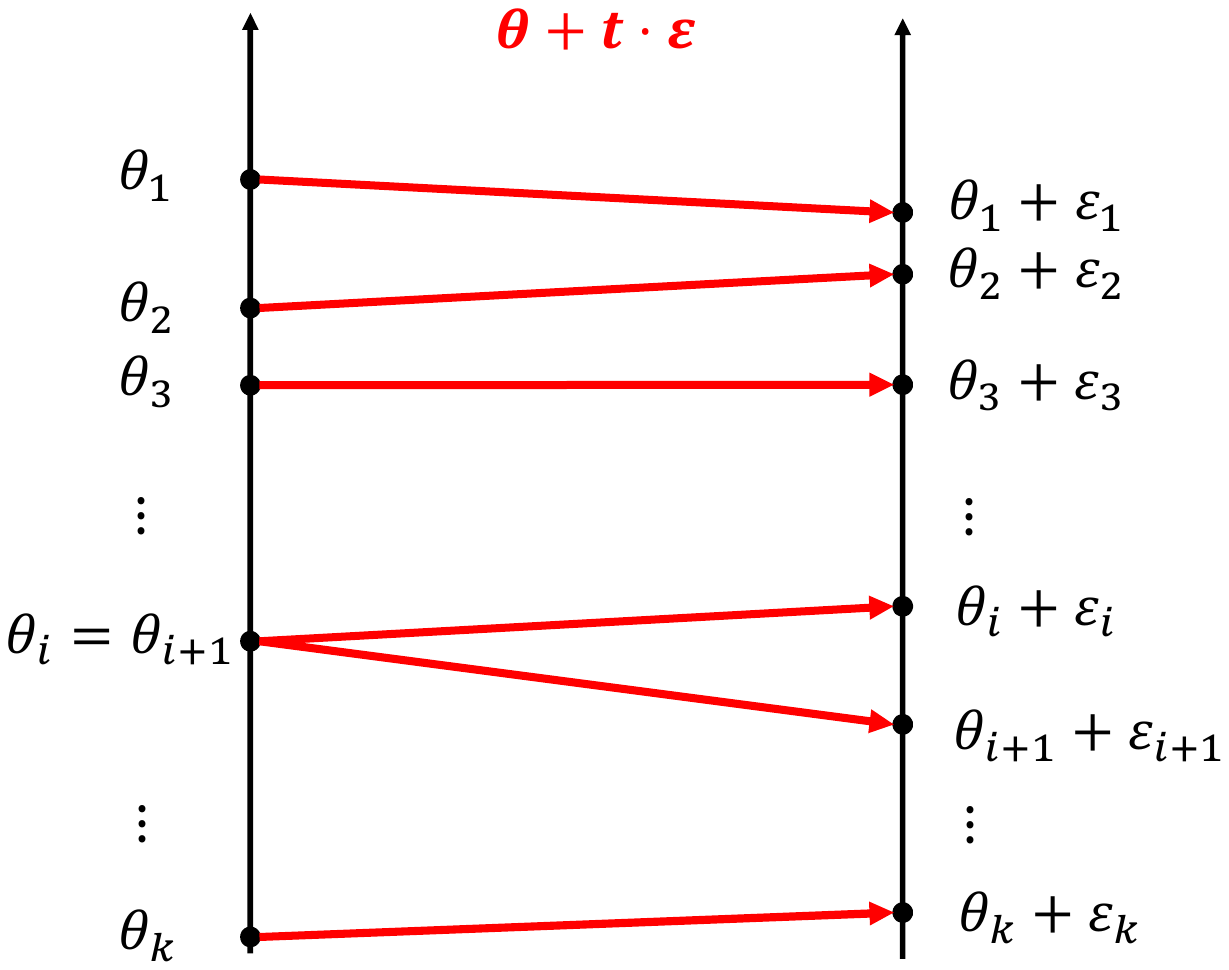}
    \end{subfigure}
    \hfill
    \begin{subfigure}[b]{0.48\textwidth}
        \centering
        \includegraphics[width=\textwidth]{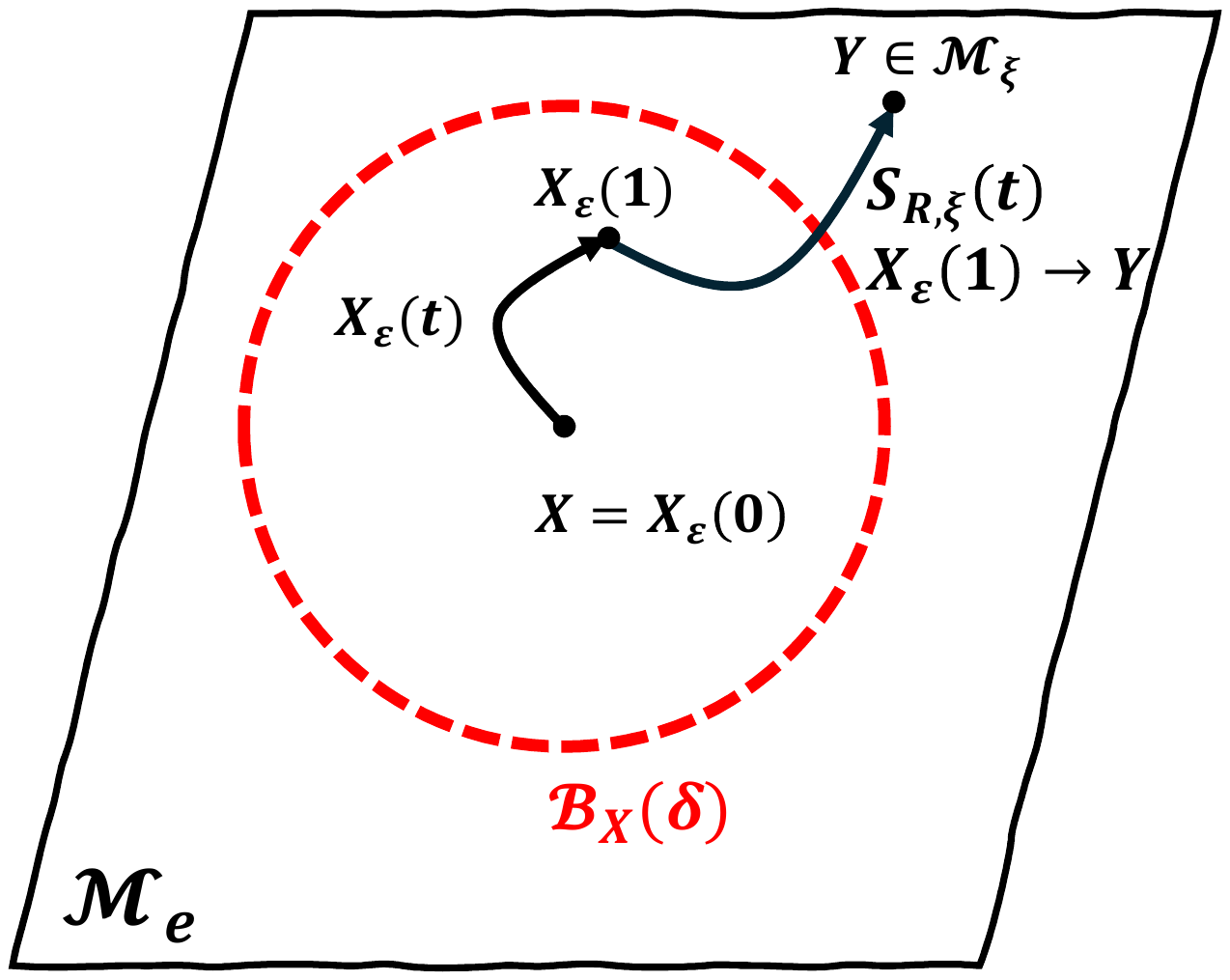}
    \end{subfigure}
    \caption{Illustration for the proof of~\cref{prop:component-inclusion}.}\label{fig:proof-of-shifts}
\end{figure}

    To see that $\mathcal{C}_{(0,0\ldots,0)}\cup \mathcal{C}_{(-1,0,\ldots,0)}\subset\mathcal{C}_*$, consider some $Q\in\so_n\setminus\frakq$ with the canonical principal angle bounded by $\theta_1 < \pi$. By~\cref{def:conjugate-locus}, all preimages of $Q\in\so_n\setminus\frakq$ are D-preimages. In view of~\cref{prop:dconnected-dpreimage}, $\theta_1$ uniquely attains the largest magnitude among principal angles. Consequently, the D-preimages in $\mathcal{C}_{*}$, see~\eqref{eq:dconnected-dpreimage}, are specified as
    \begin{equation*}
        \begin{cases}
            X = \sum_{i=1}^m R_{[i]}\sbmatrix{0 & -\theta_i\\\theta_i & 0}R_{[i]}^{\T}\\
            Y = R_{[1]}\sbmatrix{0 & -\theta_1 + 2\pi\\\theta_1 -2\pi & 0}R_{[1]}^{\T}+\sum_{i=2}^m R_{[i]}\sbmatrix{0 & -\theta_i\\\theta_i & 0} R_{[i]}^{\T}.
        \end{cases}
    \end{equation*}
    Notice that $X\in\mathcal{C}_{(0,0,\ldots, 0)}$ and, since $-\pi<\theta_1-2\pi$, $Y \in \mathcal{C}_{(-1,0,\ldots,0)}$. According to the result proven above, $\mathcal{C}_{(0,0\ldots,0)}\cup \mathcal{C}_{(-1,0,\ldots,0)}\subset\mathcal{C}_*$ follows. 
\end{proof}

In addition to the inclusion relationship between $\mathcal{C}_{\xi},\xi\in\mathbb{Z}^k$ and D-components $\mathcal{C}_e, e\in\mathcal{E}$, \cref{thm:canonical-label} further identifies D-components with canonical shifts. 

\begin{theorem}\label{thm:canonical-label}
    Consider a $\mathcal{C}_{\xi}$ constructed in~\eqref{eq:canonical-shift-dpreimage} and the D-component $\mathcal{C}_e$ that contains $\mathcal{C}_{\xi}$, i.e., $\mathcal{C}_{\xi}\subset\mathcal{C}_e$. If $\xi\neq (0,0,\ldots,0), (-1,0,\ldots, 0)$, then $\mathcal{C}_{\xi}=\mathcal{C}_e$, while $\mathcal{C}_{(0,0\ldots,0)}\cup \mathcal{C}_{(-1,0,\ldots,0)}=\mathcal{C}_*$. Consequently, the label set of D-components is identified by
    \begin{equation}\label{eq:canonical-label}
        \mathcal{E} = \{*\}\cup \Xi,\qquad \text{where $\Xi = \mathbb{Z}^k\setminus \{(0,0,\ldots,0), (-1,0,\ldots, 0)\}$.}
    \end{equation}
\end{theorem}
\begin{proof}
    It suffices to prove the reversed inclusion relationship. Consider a D-preimage $Y\in\mathcal{C}_e$ and let $\eta = (y_1,\ldots, y_k)$ be its canonical shift in $Y = \outerprodshift{R}{\theta}{y}$. Then $Y\in\mathcal{C}_{\eta}$ and~\cref{prop:component-inclusion} implies $\mathcal{C}_{\eta}\subset\mathcal{C}_e$. On the other hand, $\mathcal{C}_{\xi}\subset\mathcal{C}_e$ is assumed.
    
    Consider a $Q=R\exp(\Sigma)R^{\T}\in\so_n\setminus\frakq$ with the canonical principal angles $\sigma\in\prng^k$ under the same basis $R\in\so_n$ of $Y$, such that $\sigma$ satisfies~\eqref{eq:non-conjugate-principal-angles}. By~\cref{def:conjugate-locus}, all preimages of $Q\in\so_n\setminus\frakq$ are D-preimages. In particular, $X_{\xi}=\outerprodshift{R}{\sigma}{x}$ and $X_{\eta}=\outerprodshift{R}{\sigma}{y}$ are D-preimages. Consequently, $X_{\xi}\in\mathcal{C}_{\xi}$ and $X_{\eta}\in\mathcal{C}_{\eta}$ are D-connected in $\mathcal{C}_e$. In view of~\cref{prop:dconnected-dpreimage}, if $\xi\neq \eta$, then $X_{\xi}\neq X_{\eta}$ are only obtained in $\mathcal{C}_*$ in the form of~\eqref{eq:dconnected-dpreimage}, which implies $Y\in \mathcal{C}_{(0,0\ldots,0)}\cup \mathcal{C}_{(-1,0,\ldots,0)}$. Since $Y$ is arbitrary in $\mathcal{C}_e$, it holds that $\mathcal{C}_e\subset \mathcal{C}_{(0,0\ldots,0)}\cup \mathcal{C}_{(-1,0,\ldots,0)}\subset \mathcal{C}_*$, i.e., $\mathcal{C}_e = \mathcal{C}_* = \mathcal{C}_{(0,0\ldots,0)}\cup \mathcal{C}_{(-1,0,\ldots,0)}$. If $\xi=\eta$ and they are neither $(0,0,\ldots,0)$ nor $(-1,0,\ldots, 0)$, then $Y\in\mathcal{C}_{\xi}$ holds for all $Y\in\mathcal{C}_e$. Thus $\mathcal{C}_e\subset\mathcal{C}_{\xi}\subset\mathcal{C}_e$, i.e., $\mathcal{C}_e = \mathcal{C}_{\xi}$.
\end{proof}

\section{Diffeomorphic Logarithm and its Algorithm}\label{sec:diffeomorphic-logarithm}

Having fully characterized the D-components in $\skewm_n$, the D-preimages $X$ of $Q$ that are D-connected to a given $S\in\skewm_n\setminus\fraks$ can be completely described. If $S\in\mathcal{C}_{\xi}$ for some $\xi\in\Xi$, i.e., $S\notin\mathcal{C}_*$, then there is at most one preimage D-connected to $S$. If $S\in\mathcal{C}_*$, there are at most two D-preimages D-connected to $S$. Among these two preimages in $\mathcal{C}_*$, we adopt choosing the nearest preimage (to the reference point $S$) as the selection strategy. The choice of distances/norm for measuring is discussed in~\cref{subsec:measure-choice}. Then, the algorithm for computing the diffeomorphic logarithm and the associated subroutines are presented in~\cref{subsec:algorithms}.

\begin{definition}\label{def:diffeomorphic-logarithm}

    Let $S\in\mathcal{C}_e$ for some $e\in\cale=\{*\}\cup\Xi$ and let
    $Q\in\so_n$. Suppose that
    $\exp^{-1}(Q)\cap\mathcal{C}_e\neq\emptyset$.
    The \emph{diffeomorphic logarithm} of $Q$ at $S$ is defined by
    \[
    \mathrm{dlog}_S(Q)
    :=\underset{X\in\exp^{-1}(Q)\cap\mathcal{C}_e}{\argmin}\|X-S\|,
    \]
    where $\|\cdot\|$ is a matrix norm. If $e\in\Xi$, the minimum is attained with the only D-preimage of $Q$ in $\mathcal{C}_e$. If $e=*$ and the minimum is attained at two points, the tie is broken by choosing the only element that lies in $\mathcal{C}_{(0,0,\ldots,0)}$. If such a tie is attained with both $X$ and $Y$ in $\mathcal{C}_{(0,0,\ldots,0)}$, the diffeomorphic logarithm is not defined.
\end{definition}

Note that the diffeomorphic logarithm does not always exists, due to the obstruction in either~\cref{prop:preimage-scenario}[(i) or (iii)]. Nevertheless, \cref{coro:dcomponent-closure} guarantees that every $Q\in\so_n$ has a preimage in the closure of each $\mathcal{C}_{\xi}$ for $\xi\in\mathbb{Z}^k$ (the auxiliary $x_k=0$ when $n=2k-1$),  which provides a fallback whenever the diffeomorphic logarithm does not exist. Here, $\mathcal{C}_{\xi}$ includes those D-components with $\xi\in\Xi$ as well as those contained in $\mathcal{C}_*$ with $\xi \in \{(0,0,\ldots,0),\,(-1,0,\ldots,0)\}$.

\subsection{Choices of Matrix Norm}\label{subsec:measure-choice}

In the region $\mathcal{C}_*$, each $Q\in \so_n$ admits at most two preimages $X$ and $Y$ in the form of~\eqref{eq:dconnected-dpreimage}. The diffeomorphic logarithm is defined by selecting, relative to a reference point $S$, the closer preimage under a chosen matrix norm. While the analysis leading to this construction primarily employs the matrix $2$-norm, the selection rule itself remains valid for any matrix norm, and it induces a maximal domain of diffeomorphism relative to that choice. In practice, adopting the Frobenius norm is particularly advantageous, as it significantly simplifies the distance comparison to computing an inner product with $S$ as presented in~\cref{prop:nearest-connected-preimage}.

Recall that $R_{[i]} = \sbmatrix{R_{2i-1} & R_{2i}}$ refers to the $(2i-1)$th and $(2i)$th columns in $R$.

\begin{proposition}\label{prop:nearest-connected-preimage}
    Consider two matrices $X$ and $Y$ in $\mathcal{C}_*$ with the forms
    \begin{equation*}
        \begin{cases}
            X = \sum_{i=1}^m R_{[i]}\sbmatrix{0 & -\theta_i\\\theta_i & 0}R_{[i]}^{\T},\\
            Y = R_{[1]}\sbmatrix{0 & -\theta_1 + 2\pi\\\theta_1 -2\pi & 0}R_{[1]}^{\T}+\sum_{i=2}^m R_{[i]}\sbmatrix{0 & -\theta_i\\\theta_i & 0} R_{[i]}^{\T}.
        \end{cases}
    \end{equation*}
    For $S\in\skewm_n$, $\|X-S\|_{\F}\!<\!\|Y-S\|_{\F}$ if and only if $|\theta_1 - R_{2}^{\T}SR_{1}| \!<\! |\theta_1 -2\pi - R_{2}^{\T}SR_{1}|$.
\end{proposition}
\begin{proof}
    Both $R^{\T}XR$ and $R^{\T}YR$ are block diagonal with the skew-symmetric blocks:
    \begin{equation*}
        (R^{\T}SR)_{[i,i]} = R_{[i]}^{\T}SR_{[i]} = \sbmatrix{0 & R_{2i-1}^{\T}SR_{2i}\\R_{2i}^{\T}SR_{2i-1} &0}, i\leq m.
    \end{equation*}
    Thus, the two equalities $\|X-S\|_{\F}^2 = \sum_{i\neq j}\|(R^{\T}SR)_{[i,j]}\|_{\F}^2 + \sum_{i=1}^m 2(\theta_i - R_{2i}^{\T}SR_{2i-1})^2$ and $\|Y-S\|_{\F}^2 = \sum_{i\neq j}\|(R^{\T}SR)_{[i,j]}\|_{\F}^2 + 2(\theta_1 -2\pi  - R_{2}^{\T}SR_{1})^2+ \sum_{i=2}^m 2(\theta_i - R_{2i}^{\T}SR_{2i-1})^2$ hold and the desired result follows.
\end{proof}

Any two matrix norms $\|\cdot \|_{a}$ and $\|\cdot\|_b$ are equivalent in the sense that there exists $r, s > 0$, such that $r\|M\|_a\leq \|M\|_b\leq s\|M\|_a$ holds for all $M\in\realmat{n}$ (see, e.g.,~\cite{golub2013matrix} and~\cite{horn2012matrix}). In particular, if there exists a D-preimage $X$ of $Q$ that satisfies $\|X-S\|_2 < \pi$, the bounds between the matrix Frobenius norm, $2$-norm and max-norm ensure that $X$ is the nearest preimage to $S$ in both Frobenius norm and $2$-norm. This is formalized in~\cref{lemma:matrix-norm-bound} and~\cref{def:nearest-frobenius}.

\begin{lemma}\label{lemma:matrix-norm-bound}
    For $M = [M_{i,j}]_{i,j\leq n}\in\realmat{n}$, its matrix $2$-norm is bounded by the Frobenius norm $\|M\|_{\F} := \sqrt{\sum_{i,j}^n M_{i,j}^2}$ and $\|M\|_{\max}:= \max_{i,j\leq n}|M_{i,j}|$ in the form of 
    \begin{equation}\label{eq:matrix-norm-bound}
        \|M\|_{\max} \leq \|M\|_2 \leq \|M\|_{\F}.
    \end{equation}
\end{lemma}

\begin{definition}\label{def:nearest-frobenius}
    The Frobenius-nearest preimage of $Q \in \so_n$ to $S \in \skewm_n$ is
    \begin{equation}\label{eq:frobenius-nearest}
        \flog_S(Q) := \underset{X\in\exp^{-1}(Q)}{\arg\min}\|X-S\|_{\F}.
    \end{equation} 
    The notion of nearest preimage is defined over the entire set of preimages $\exp^{-1}(Q)$ and does not depend on the diffeomorphism structure. Thus, the feasible set of $X$ is not restricted to a prescribed D-component $\mathcal{C}_e$.
\end{definition}

If a Schur basis $R\in \orth_n$ of $X\in\flog_S(Q)$ is provided, then $X$ can be determined by a simple rounding-to-the-nearest process as presented in \cref{prop:frobenius-shift-with-schur}.

\begin{proposition}\label{prop:frobenius-shift-with-schur}
    If $X = \sum_{i=1}^m R_{[i]}\sbmatrix{0 & -\theta_i-2\pi x_i\\\theta_i+2\pi x_i & 0} R_{[i]}^{\T}\in \flog_S(Q)$, the angular shift $\xi = (x_1,\ldots, x_k)$ is given by
    \begin{equation}\label{eq:frobenius-shift}
        x_i \in \mathrm{RTN}\Big(\frac{R_{2i}^{\T}S R_{2i-1}-\theta_i}{2\pi}\Big), \text{ $i\leq m$; the auxiliary $x_k = 0$ when $n=2k-1$.}
    \end{equation} 
    Here, $\mathrm{RTN}$ refers to rounding to the nearest integer(s).
\end{proposition}
\begin{proof}
    Notice that $\|X - S\|_{\F} = \|R^{\T}XR - R^{\T}SR\|_{\F} $ where $R^{\T}XR$ is block diagonal with $\sbmatrix{0 & -\theta_i -2\pi x_i\\ \theta_i + 2\pi x_i & 0}, i\leq m$ on the diagonal, see~\cref{def:angles-and-shifts}. Denote $S' = R^{\T}SR$ and the diagonal block $S'_{[i,i]}, i\leq m$ is given by 
    \begin{equation*}
        S'_{[i,i]} = R_{[i]}^{\T}S R_{[i]} = \sbmatrix{R_{2i-1}^{\T}\\R_{2i}^{\T}}S\sbmatrix{R_{2i-1} &R_{2i}} = \sbmatrix{R_{2i-1}^{\T}SR_{2i-1} & R_{2i-1}^{\T}SR_{2i}\\ R_{2i}^{\T}SR_{2i-1} & R_{2i}^{\T}SR_{2i}}
    \end{equation*}
    where $R_{2i-1}^{\T}SR_{2i} = -R_{2i}^{\T}SR_{2i-1}$ and the zeros on the diagonal of $S'$ follow from the skew-symmetry of $S'$, and when $n = 2k - 1$, $(R^{\T}XR)_{[k,k]} = S'_{[k,k]} = \sbmatrix{0}$. The following equalities then hold
    \begin{equation*}
        \begin{aligned}
            \|R^{\T}XR-S'\|_{\F}^2 &=\textstyle \sum_{i\neq j}^k \|S'_{[i,j]}\|_{\F}^2 + \sum_{i=1}^m \Big\|\sbmatrix{0 & -\theta_i-2\pi x_i\\\theta_i+2\pi x_i}- S'_{[i,i]}\Big\|_{\F}^2\\
            &= \textstyle \sum_{i\neq j}^k \|S'_{[i,j]}\|_{\F}^2 + \sum_{i=1}^m 2(\theta_i +2\pi x_i - R_{2i}^{\T}SR_{2i-1})^2.
        \end{aligned}
    \end{equation*}
    Therefore, $X$ minimizes~\eqref{eq:frobenius-nearest} if and only if $\xi$ minimizes
    \begin{equation*}
        \underset{x_1,\ldots, x_m\in\mathbb{Z}}{\argmin} \textstyle \sum_{i=1}^m |\theta_i + 2\pi x_i - R_{2i}^{\T}SR_{2i-1}|,
    \end{equation*}
    which admits the closed-form solution(s) given by rounding to the nearest integer(s): $x_i\in \mathrm{RTN}\left(\left(R_{2i}^{\T}S R_{2i-1}-\theta_i\right)\big/ 2\pi\right)$ for $i\leq m$.
\end{proof}

\begin{remark}
    If $X\in\flog_S(Q)\cap (\skewm_n\setminus\fraks)$, then any Schur basis $R$ of $Q$ determines $X$ with the shift(s) solved from~\eqref{eq:frobenius-shift}, because any Schur basis of $Q$ is also a Schur basis of its D-preimage, see~\cref{prop:shared-basis}. 
\end{remark}

Notice that~\eqref{eq:frobenius-shift} identifies the shift(s) such that $\theta_i+2\pi x_i$ that is (are) within $\pi$ interval of $R_{2i}^{\T}S R_{2i-1}$, suggesting that the matrix max-norm of the difference $\|R^{\T}XR - S'\|$ is given by either $\max_{i\neq j}\{\|S'_{[i,j]}\|_{\max}\}$ or $\max_{i\leq m}\{|\theta_i+2\pi x_i - R_{2i}^{\T}S R_{2i-1}|\}$, where the latter is bounded by $\pi$. This observation leads to the the next proposition relating the nearest Frobenius preimage and the nearest D-preimage in matrix $2$-norm.

\begin{proposition}\label{prop:flog-to-nlog}
    Consider $Q\in\so_n$ and $S\in\skewm_n$. If $X$ is a D-preimage of $Q$ satisfying $\|X-S\|_2 < \pi$, then $X$ is the unique solution of~\eqref{eq:frobenius-nearest}.
\end{proposition}
\begin{proof}
    Suppose $X$ is a D-preimage that satisfies $\|S-X\|_2 < \pi$. Let $R$ be a Schur basis of $Q$ and $X = \outerprodskew{R}{\alpha} := RAR^{\T}$, see~\cref{def:angles-and-shifts} for notation convention, the bound in~\eqref{eq:matrix-norm-bound} yields 
    \begin{equation*}
        \pi > \|S-X\|_2 = \|R^{\T}SR - A\|_2\geq \|R^{\T}SR - A\|_{\max} \geq \max_{i\leq m}|R_{2i}^{\T}SR_{2i-1} - \alpha_i|.
    \end{equation*}
    The bounds on $\pi > |R_{2i}^{\T}SR_{2i-1} - \alpha_i|$ for all $i\leq m$ yield the same shift in~\eqref{eq:frobenius-shift} from the RTN. Moreover, the strict bound on $\pi$ implies that the shift is unique regardless of the choice of Schur basis $R$. Finally, notice that the Schur basis $R$ is arbitrarily chosen from $Q$, while any Schur basis of the preimage of $Q$ remains a Schur basis of $Q$. Thus $X$ is the minimizer of~\eqref{eq:frobenius-nearest} by a unique shift.
\end{proof}

Note that~\cref{prop:flog-to-nlog} does not apply to the non-D-preimages, because the Schur basis of $Q$ is not necessarily a Schur basis of its non-D-preimages.

\begin{corollary}\label{coro:nlog-necessary}
    If the nearby matrix logarithm, see~\cref{def:nearby-logarithm}, $\nlog_S(Q)$ of $Q$ around $S\in\mathcal{C}_e$ exists, i.e., $\nlog_S(Q)\in \mathcal{C}_e\cap \{X:\|X-S\|_2<\pi\}\cap \exp^{-1}(Q)$, it is the unique solution of~\eqref{eq:frobenius-nearest}.
\end{corollary}
\begin{proof}
    Since $\nlog_S(Q)$ is a D-preimage and $\|\nlog_S(Q)-S\|_2 < \pi$, \cref{coro:nlog-necessary} follows.
\end{proof}

\begin{figure}
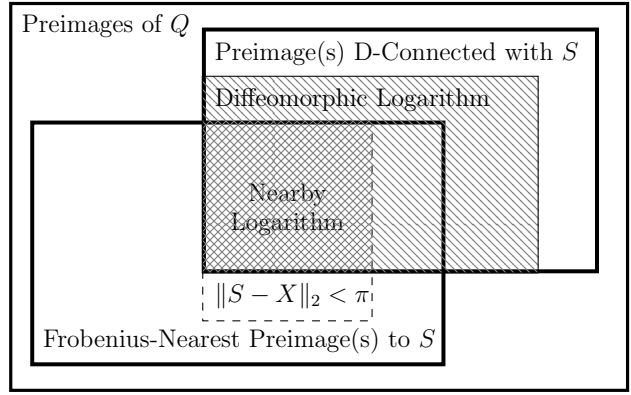

    \centering
    \logVenn[0.8]
    \caption{Venn Diagram of a Preimage $X$ of $Q$ with $S\in\skewm_n\setminus\fraks$}\label{fig:venn}
\end{figure}

\begin{remark}\label{remark:venn}
    When the diffeomorphic logarithm adopts the matrix Frobenius norm as selection rules over two preimages in $\mathcal{C}_*$, \cref{prop:flog-to-nlog} and \cref{coro:nlog-necessary} yield the Venn diagram~\cref{fig:venn}. Indeed, when (i) $S\notin\mathcal{C}_*$, the D-connectivity from $X$ to $S$ is the only criterion of whether $X$ is the diffeomorphic logarithm, whether or not $X$ is a nearest Frobenius preimage. When $S\in\mathcal{C}_*$ and $X$ is D-connected to $S$ (i.e., $X\in\mathcal{C}_*$), (ii) if $\|X-S\|_2<\pi$, then $X$ is the nearby logarithm, the nearest Frobenius preimage and the diffeomorphic logarithm; (iii) if $\|X-S\|_2>\pi$ but it remains the nearest Frobenius preimage, then $X$ is still the diffeomorphic logarithm; (iv) if $X$ is not a nearest Frobenius preimage, then $X$ is not the nearby logarithm. 
\end{remark}

The Venn diagram~\cref{fig:venn} is intended as an illustration of logical relationships, rather than a partition where every region is necessarily nonempty. For example, when $S\notin\mathcal{C}_*$, the ``diffeomorphic logarithm'' box coincides with the ``D-connected'' box. We keep the conditions represented as distinct boxes because they genuinely define different subsets.

\subsection{Algorithms and Computations}\label{subsec:algorithms}

With the D-components labeled by the $2\pi$-shifts of the canonical angles, computing D-preimages reduces to computing a canonical Schur decomposition of $Q\in\so_n$ together with determining the shift $\xi\in\mathbb{Z}^k$. While Schur or eigenvalue decompositions for general matrices are well understood and highly optimized, the special orthogonal structure can be further exploited to accelerate the computation.

Mataigne and Gallivan~\cite{mataigne2024eigenvalue} propose an efficient Schur decomposition algorithm for normal matrices $\{X\in\realmat{n}:X^{\T}X = XX^{\T}\}$, a class that includes $\so_n$. The algorithm requires $14/3\,n^3$ operations, of which $2n^3$ are Level-3 \texttt{BLAS} operations. Consider a normal matrix in the block-diagonal real Schur form \eqref{eq:block-diagonal-normal}:
\begin{equation*}
    M = RTR^{\T} = \sum_{i=1}^k R_{[i]}T_{[i,i]}R_{[i]}^{\T},\, \text{where}\,\begin{cases}
        T_{[i,j]} = \zerov, & i\neq j,\\
        T_{[i,i]} = \sbmatrix{r_i & -s_i\\ s_i & r_i}, & i\leq m,\\
        T_{[k,k]} = r_k, & k = m+1,
    \end{cases}
\end{equation*}
Here, $R$ and $T$ are the factors to be computed. Algorithm~3.1 (Step 1--3) of~\cite{mataigne2024eigenvalue} first computes the Schur basis $R$ of the skew-symmetric part of the normal matrix,
\begin{equation*}
     \textstyle \mathrm{skew}(M) = R\,\mathrm{skew}(T)R^{\T} = \sum_{i=1}^m R_{[i]}\sbmatrix{0 & -s_i\\ s_i & 0} R_{[i]}^{\T},
\end{equation*}
and then determines the real coefficients $r_i$, $i\leq k$, in subsequent steps of the algorithm.

For skew-symmetric matrices, it follows immediately that $r_i=0$. For special orthogonal matrices, the relation $r_i^2+s_i^2=1$ implies $|r_i|=\sqrt{1-s_i^2}$ for all $i\leq m$, so it remains to determine the sign of $r_i$. Writing $R_{[i]}=\sbmatrix{R_{2i-1} & R_{2i}}$ for $i\leq m$, we have
\begin{equation*}
    QR_{[i]} = R_{[i]}E_{[i,i]} = \sbmatrix{R_{2i-1} & R_{2i}} \sbmatrix{r_i & -s_i\\s_i & r_i} = \left[{\footnotesize
    \begin{array}{c|c}
    r_i R_{2i-1} + s_i R_{2i} & -s_iR_{2i-1} + r_iR_{2i}
    \end{array}
    }\right].
\end{equation*} 
Hence, the sign\footnote{Comparing the first nonzero entries between $r_i R_{2i-1}$ and $QR_{2i-1}-s_iR_{2i}$ is often used as an efficient alternative for the sign.} of $r_i$ is determined from $r_i R_{2i-1}=QR_{2i-1}-s_iR_{2i}$, namely,
\begin{equation*}
    \mathrm{sign}(r_i) = \frac{R_{2i-1}^{\T}(QR_{2i-1} - s_iR_{2i})}{\|R_{2i-1}\|\|QR_{2i-1} - s_iR_{2i}\|}
\end{equation*}
where $\|\cdot\|$ denotes the Euclidean norm. As a computationally cheaper but slightly less robust alternative, $\mathrm{sign}(r_i)$ may also be obtained from the first nonzero components of $R_{2i-1}$ and $QR_{2i-1}-s_iR_{2i}=r_iR_{2i-1}$, thereby avoiding the $O(n)$ operations required to evaluate inner products and norms.

A specialized Schur decomposition for skew-symmetric or special orthogonal matrices is given in \cref{alg:so-schur}. The lines between \ref{soschur-alg:skew-start} and \ref{soschur-alg:skew-end} are adapted from~\cite[Algorithm~3.1 (steps~1--3)]{mataigne2024eigenvalue}. In that procedure, the determinant of the resulting Schur basis $R$ is not determined and is not exploited in the original implementation. However, this determinant is essential for constructing a canonical Schur decomposition in~\cref{prop:existence-canonical-decomp}, and computing it from scratch would require $O(n^3)$ operations. Therefore, \cref{alg:so-schur} introduces the formulae in line~\ref{soschur-alg:determinant-b} and line~\ref{soschur-alg:determinant-c}, which are justified in~\cref{prop:schur-basis-determinant}.

\begin{algorithm}
    \caption{Schur decomposition of a skew-symmetric or special orthogonal matrix}\label{alg:so-schur}
    \begin{algorithmic}[1]
        \REQUIRE Skew-symmetric or special orthogonal $M\in\so_n$.
        \ENSURE $M = RAR^{\T}\in\skewm_n$ with $A_{[i,i]} = \sbmatrix{0 & -\alpha_i\\ \alpha_i & 0}$ for $i\leq m$, or \\$M = R\exp(\Theta)R^{\T}\in\so_n$ with $\Theta_{[i,i]} = \sbmatrix{0 & -\theta_i\\ \theta_i & 0}, \theta_i\in \prng$ for $i\leq m$.
        \STATE\label{soschur-alg:skew-start} $S \gets (M - M^{\T})/2$
        \STATE (Hessenberg transformation) $S = H\sbmatrix{0 & -a_1 & 0 & \cdots\\ a_1 & 0 & -a_2 & \cdots\\ 0 & a_2 & 0 & \cdots\\\vdots & \vdots &\vdots &\ddots}H^{\T}$\label{soschur-alg:hessenberg}  \codecomment{Level 3 \texttt{BLAS}}
        \STATE (Permutation $P$) $S = HP \sbmatrix{\zerov_{k\times k} & -B^{\T}\\ B & \zerov_{m\times m}}P^{\T}H^{\T}$, where $B$ is $m\times k$ bidiagonal in the forms of
        \begin{equation*}
            \sbmatrix{a_1 & -a_2 & \cdots & 0 & 0\\ 0 & a_3 & \cdots & 0 & 0\\ \vdots &\vdots &\ddots &\vdots &\vdots\\ 0 & 0 &\cdots & a_{2m-3} & -a_{2m-2}\\ 0 & 0 &\cdots & 0 & a_{2m-1}}\, \text{or}\, \sbmatrix{a_1 & -a_2 & \cdots & 0 & 0 & 0\\ 0 & a_3 & \cdots & 0 & 0 & 0\\ \vdots &\vdots &\ddots &\vdots &\vdots &\vdots \\ 0 & 0 &\cdots & a_{2m-3} & -a_{2m-2} & 0\\ 0 & 0 &\cdots & 0 & a_{2m-1} & -a_{2m}}
        \end{equation*}
        \IF{$n = 2m$}
            \STATE (Singular Value Decomposition) $B = U\Sigma V^{\T}$
            \STATE $R\gets HP\sbmatrix{U &\zerov\\\zerov & V}P^{\T}$ and $(s_1,\ldots, s_m)\gets$ the diagonals in $\Sigma$ 
        \ELSIF{$ n =2m+1$}
            \STATE (Givens Elimination $G$) $B = G\sbmatrix{C & \zerov_{m\times 1}}$ where $C$ is $m\times m$ bidiagonal
            \STATE (Singular Value Decomposition) $C = U\Sigma V^{\T}$
            \STATE $R\gets HP\sbmatrix{U &\zerov&\zerov\\\zerov & 1 &\zerov \\\zerov &\zerov& GV}P^{\T}$ and $(s_1,\ldots, s_m)\gets$ the diagonals in $\Sigma$
            \STATE $s_k\gets 0$ and $r_k\gets 1$.
        \ENDIF\label{soschur-alg:skew-end}
        \IF{$n=2m$}
            \STATE $\det(R)\gets (-1)^n \prod_{i=1}^m \mathrm{sign}(B_{i,i})$\label{soschur-alg:determinant-b}
        \ELSE
            \STATE $\det(R)\gets (-1)^n \prod_{i=1}^m \mathrm{sign}(C_{i,i})$\label{soschur-alg:determinant-c}
        \ENDIF 
        \IF{$M\in\skewm_n$}
            \STATE Return $R$, $\det(R)$ and $\alpha = (s_1,\ldots, s_k)$.
        \ELSE
            \FOR{$i = 1,\ldots, m$}\label{soschur-alg:sign-start}
                \STATE $u\gets R_{2i-1}$ and $v\gets QR_{2i-1} - s_iR_{2i}$
                \STATE $r_i \gets \frac{u^{\T}v}{\|u\|\|v\|}\sqrt{1-s_i^2}$\codecomment{Get the sign of $r_i$ from $v = r_i u$}
                \STATE $\theta_i\gets \arctan(s_i, c_i)\in \prng$
            \ENDFOR
            \RETURN \label{soschur-alg:sign-end} $R$, $\det(R)$ and $\theta = (\theta_1,\ldots, \theta_k)$.
        \ENDIF
    \end{algorithmic}
\end{algorithm}

\begin{proposition}\label{prop:schur-basis-determinant}
    Let $R\in\orth_n$ be obtained from~\cref{alg:so-schur} along with the intermediate matrices: 
    \begin{equation}
        H, P\in \orth_n, U, V, G\in\orth_m\; \text{and}\; \text{and the $m\times m$ bidiagonal $B$ or $C$.}
    \end{equation}
    If the diagonals of $B$ or $C$ are nonzeros, then
    \begin{equation}\label{eq:schur-basis-determinant}
        \det(R) = \begin{cases}
        (-1)^n\prod_{i=1}^m \mathrm{sign}(B_{i,i}) & n = 2m,\\
        (-1)^n\prod_{i=1}^m \mathrm{sign}(C_{i,i}) & n = 2m+1.
    \end{cases}
    \end{equation}
\end{proposition}
\begin{proof}
    The Hessenberg transformation $H$ consists of $n-2$ elementary Householder reflectors, which yields $\det(H) = (-1)^{n-2} = (-1)^{n}$. The Givens elimination matrix $G$ consists of elementary Givens rotations is in $\so_m$, i.e., $\det(G) = 1$, as each elementary Givens rotation is in $\so_m$. Consider a singular value decomposition $U\Sigma V^{\T} = X\in\realmat{m}$ with nonzero diagonals $X_{i,i}\neq 0$ for $i\leq m$ in $X$. The nonzero diagonals yields $\det(X)\neq 0$ and $\det(\Sigma) = |\det(X)| > 0$. Consequently,
    \begin{equation*}
        \det(U)\det(V) = \det(X)/\det(\Sigma) = \mathrm{sign}(\det(X)) = \textstyle \prod_{i=1}^m \mathrm{sign}(X_{i,i}),
    \end{equation*}
    while $\det(P)\det(P^{\T}) = \det(P)^2$. Thus $\det(R)$ is reduced to $\det(H)\det(U)\det(V^{\T})$ where 
    \begin{equation*}
        \begin{cases}
            \det(R) = \det(H)\,\mathrm{sign}(\det(B))& n = 2m,\\
            \det(R) = \det(H)\,\mathrm{sign}(\det(C)) & n = 2m+1.\\
        \end{cases}
    \end{equation*}
    which yields~\eqref{eq:schur-basis-determinant}.
\end{proof}


In view of~\cref{prop:existence-canonical-decomp}, for any Schur basis $R\in\orth_n$ and angle vector $\alpha\in\realset^k$ (the auxiliary $\alpha_k = 0$ when $n = 2k-1$), the canonical Schur decomposition of $Q = V\exp(\Theta)V^{\T}$ can be obtained by shuffling and
sign-flipping the columns of $R$, as described in~\cref{alg:canonical-alignment}. The determinant computation in line~\ref{canschur-alg:determinant} of \cref{alg:canonical-alignment} requires $O(n^3)$ operations in general. This cost can be avoided when the Schur basis is computed via \cref{alg:so-schur}.

\begin{algorithm}
    \caption{Canonical alignment of angles}\label{alg:canonical-alignment}
    \begin{algorithmic}[1]
        \REQUIRE $R\in\orth_n$, (optional $\det(R)$) and $\alpha\in\realset^k$ (the auxiliary $\alpha_k = 0$ when $n = 2k-1$). 
        \ENSURE $VBV^{\T} = RAR^{\T}$ with $\beta = \sigma + 2\pi \eta$ where $V\in\so_n$ and $\sigma$ satisfies~\eqref{eq:canonical-angles}. 
        \FOR{$i = 1,\ldots, m$}
            \STATE $x_i\gets \lfloor (\alpha_i + \pi) \big/ 2\pi\rfloor$ \codecomment{$\lfloor\,\cdot\, \rfloor$ : floor operator}
            \STATE $\theta_i\gets \alpha_i - 2\pi x_i$
        \ENDFOR
        \STATE $(\sigma, \eta, V) \gets (\theta, \xi, R)$
        \IF{$\det(R) = -1$}\label{canschur-alg:determinant}
            \STATE $(\sigma_1, y_1, V_{[1]}) \gets (-\sigma_1, -y_1, \sbmatrix{V_2 & V_1})$
        \ENDIF
        \STATE (Sort by magnitude $\rho$) $|\sigma_{\rho(i)}|\geq |\sigma_{\rho(i+1)}|$ (the auxiliary $\rho(k) = k$ for $n = 2k-1$).
        \STATE $(\sigma, \eta, V)\gets \left((\sigma_{\rho(1)},\ldots, \sigma_{\rho(k)}), (y_{\rho(1)},\ldots, y_{\rho(k)}), \sbmatrix{V_{[\rho(1)]} &\cdots & V_{[\rho(k)]}}\right)$
        \FOR{$i = 1,\ldots, m-1$}
            \IF{$\sigma_i < 0$}
                \STATE $(\sigma_i, y_i, V_{[i]})\gets (-\sigma_i, -y_i, \sbmatrix{V_{2i} & V_{2i-1}})$
                \STATE $(\sigma_{i+1}, y_{i+1}, V_{[i+1]})\gets (-\sigma_{i+1}, -y_{i+1}, \sbmatrix{V_{2i+2} & V_{2i+1}})$
            \ENDIF
        \ENDFOR
        \IF{$n = 2m+1$ and $\sigma_m < 0$}
            \STATE $(\sigma_m, y_m, V_{[m]})\gets (-\sigma_m, -y_m, \sbmatrix{V_{2m} & V_{2m-1}})$
            \STATE $V_{[m+1]}\gets -V_{[m+1]}$ \codecomment{Note that $V_{[m+1]} = V_n$}
        \ENDIF
        \RETURN $\sigma\in\prng^k$ subject to~\eqref{eq:canonical-angles}, $\eta\in\mathbb{Z}^k$ and $V\in\so_n$.
    \end{algorithmic}
\end{algorithm}

The computation of the diffeomorphic logarithm is summarized in \cref{alg:diffeomorphic-logarithm}. First, a specialized Schur decomposition (\cref{alg:so-schur}) is applied to the reference point $S\in\skewm_n$, and the canonical alignment procedure (\cref{alg:canonical-alignment}) determines the canonical shift $\xi$ that locates the D-component containing $S$. A second Schur decomposition and canonical alignment are then performed on $Q\in\so_n$ to identify a preimage $X$ lying in $\mathcal{C}_{\xi}$ or in its closure. 

In line~\ref{dlog-alg:frobenius-check}, the Frobenius norm is used to select among candidate preimages in the special component $\mathcal{C}_*$; this criterion may be replaced by the matrix $2$-norm at higher computational cost. Line~\ref{dlog-alg:2pi-check} provides an optional distance check that determines whether the resulting diffeomorphic logarithm coincides with the nearby logarithm. Also note that the rank-deficiency check in line~\ref{dlog-alg:conjugate-check} involves with examining exactly canceled floating point numbers: $|\sigma_i\pm\sigma_j|= 0$, which is typically implemented by $|\sigma_i+\sigma_j|\leq \varepsilon$ with sufficiently small $\varepsilon$. For the integer type $y_i\pm y_j = 0$, the exact integer arithmetic is enough for the verification.

\begin{algorithm}
    \caption{Diffeomorphic Logarithm of $Q$}\label{alg:diffeomorphic-logarithm}
    \begin{algorithmic}[1]
        \REQUIRE $Q\in\so_n$ and $S\in\skewm_n\setminus\fraks$. 
        \ENSURE A preimage $X$ of $Q$ in the closure of $\mathcal{C}_{\xi}\ni S$.
        \STATE (Schur decomposition) $(R, \det(R), \alpha)\gets$~\cref{alg:so-schur} with $S$
        \STATE (Schur decomposition) $(V, \det(V), \sigma)\gets$~\cref{alg:so-schur} with $Q$
        \STATE (Canonical Alignment) $(\theta, \xi, R)\gets$~\cref{alg:canonical-alignment} with $(R, \det(R), \alpha)$
        \STATE (Canonical Alignment) $(\sigma, \eta, V)\gets$~\cref{alg:canonical-alignment} with $(V, \det(V), \sigma)$
        \IF{$\xi\neq (0,0,\ldots, 0), (-1,0,\ldots, 0)$}\label{dlog-alg:frobenius-check}
            \STATE $\beta \gets \sigma + 2\pi\xi$
        \ELSE 
            \IF{$|\sigma_1-V_2^{\T}SV_1| > |\sigma_1-2\pi -V_2^{\T}SV_1|$}
                \STATE $\beta \gets (\sigma_1-2\pi,\sigma_2,\ldots,\sigma_k)$
            \ELSE
                \STATE $\beta \gets \sigma$
            \ENDIF
        \ENDIF
        \STATE $X\gets \sum_{i=1}^mV_{[i]}\sbmatrix{0 & -\beta_i\\\beta_i & 0}V_{[i]}^{\T}$.
        \STATE Label $X$ as ``the diffeomorphic logarithm of $Q$ with $S$''
        \IF{(Optional) $\|X-S\|_2<\pi$}\label{dlog-alg:2pi-check}
            \STATE Re-label $X$ as ``the nearby logarithm of $Q$ around $S$''
        \ENDIF
        \FOR{$i\neq j\leq k$}
            \IF{$|\sigma_i+\sigma_j| = 0$ but $y_i\neq y_j$ or $|\sigma_i-\sigma_j| = 0$ but $y_i\neq -y_j$\label{dlog-alg:conjugate-check}}
                \STATE Re-label $X$ as ``a non-D-preimage of $Q$ in the closure of $\mathcal{C}_{\xi}$''.
            \ENDIF
        \ENDFOR
        \RETURN $X$ with its label. (Optional) canonical Schur info: $V$ and $\beta$.
    \end{algorithmic}
\end{algorithm}

\section{Numerical Experiments}\label{sec:experiments}

This section presents numerical experiments designed to evaluate the computational efficiency, numerical accuracy, and practical implications of the proposed diffeomorphic logarithm on the special orthogonal group. The first experiment compares the diffeomorphic logarithm with the generic principal logarithm computed by the built-in \texttt{MATLAB} routine \texttt{logm}, highlighting differences in runtime, structure preservation, and reconstruction error. The second experiment investigates the impact of diffeomorphic logarithms on the Karcher mean problem with time-varying data, demonstrating the continuity advantages provided by diffeomorphic logarithms relative to principal logarithms. Together, these studies illustrate both the numerical benefits and geometric relevance of the proposed framework.

\subsection{Computing Environment}
All experiments were conducted in \texttt{MATLAB} (R2024a). \Cref{alg:so-schur}, \cref{alg:canonical-alignment} and~\cref{alg:diffeomorphic-logarithm} were implemented in \texttt{C++} and executed in \texttt{MATLAB} via the \texttt{MEX gateway}. \Cref{alg:so-karcher-mean} was implemented and executed in \texttt{MATLAB}. The source code used in this section is publicly available at \url{https://github.com/zhifeng1703/cpp-released-code/diffeoLog}.
{\footnotesize
\begin{itemize}
    \item CPU: \texttt{Intel(R) Core(TM) i7-1065G7, 4 Cores, 1.30 GHz}.
    \item RAM: \texttt{32 GB}
    \item OS: \texttt{Windows 10}.
    \item BLAS and LAPACK Library: \texttt{MATLAB} built-in \texttt{BLAS} and \texttt{LAPACK} (\texttt{Intel(R) MKL}).
\end{itemize}
}

\subsection{Compute Time and Error}

\Cref{fig:logm-numerics} reports the computational cost (measured in wall-clock time) and accuracy (measured in the root mean square error, RMSE) of the diffeomorphic logarithm in~\cref{alg:diffeomorphic-logarithm} and the \texttt{MATLAB} built-in \texttt{logm} routine. For random $Q\in\so_n$ and reference $S\in\skewm_n$, we compute the diffeomorphic logarithm $\mathrm{dlog}_S(Q)$ with and without the canonical shift $\xi$ of $S$, the latter requiring an additional Schur decomposition of $S$. In terms of computational cost, the diffeomorphic logarithm with known shift is consistently the fastest, achieving a speedup of about $10$ over the built-in \texttt{logm}. The unknown-shift case incurs a modest overhead due to the additional Schur decomposition, but still yields a speedup of several times compared to \texttt{logm}. 

All Schur-based logarithms of the form~\eqref{eq:dpreimage-shifted-angle}, including the diffeomorphic logarithm, inherit their dominant numerical error from the Schur decomposition of $Q$, whose stability and accuracy have been carefully studied in~\cite{mataigne2024eigenvalue}. In our tests, the eigenvalue gap drops below $10^{-4}$ when $n\geq 30$, and the Schur decomposition is expected to be more sensitive. Nevertheless, the Schur-based logarithms remain accurate for the tested dimensions: while MATLAB's built-in \texttt{logm} achieves a consistently smaller RMSE below $10^{-15}$, our implementation maintains an RMSE below $10^{-13}$ for $n<200$. This accuracy is adequate for the intended computations, especially considering that the Schur-based logarithm is about an order of magnitude faster in these tests. For applications requiring higher accuracy, the error of a Schur-based logarithm can be further reduced by employing a more accurate Schur decomposition procedure.

\begin{figure}[tbp]
    \centering
    \begin{subfigure}[b]{0.48\textwidth}
        \centering
        \includegraphics[width=\textwidth]{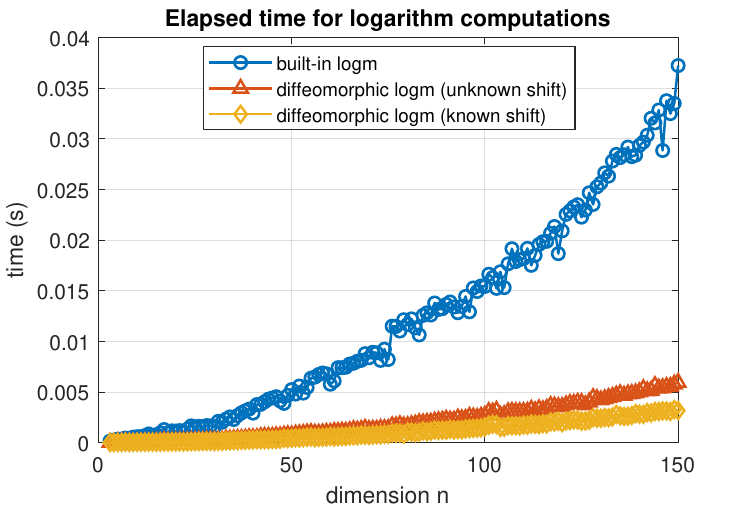}
    \end{subfigure}
    \hfill
    \begin{subfigure}[b]{0.48\textwidth}
        \centering
        \includegraphics[width=\textwidth]{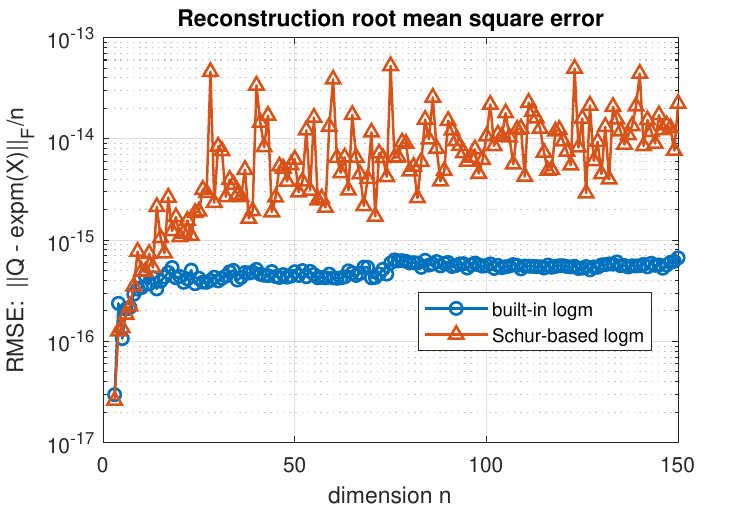}
    \end{subfigure}
    \caption{Compute time and error of different logarithms}\label{fig:logm-numerics}
\end{figure}

\subsection{Karcher Mean with Moving Data}

Beyond pointwise logarithm computation, logarithmic maps in $\so_n$ play a central role in optimization problems in matrix manifolds. In particular, the Karcher mean provides a notion of averaging for rotational data and serves as a representative example where continuity of the logarithm directly affects the behavior of iterative algorithms. This experiment illustrates how the choice of logarithmic inverse influences the evolution of the mean.

Given data points $Q_i$ in $\so_n$ equipped with the canonical metric $g$ proposed in~\cite{edelman1998}, the Karcher mean is defined as the minimizer of the sum of squared distances,
\begin{equation}\label{eq:karcher-mean}
    \underset{M\in\so_n}{\arg\min}\textstyle \frac{1}{2K}\sum_{i=1}^K {\rm dist}(M,Q_i)^2,
\end{equation}
where ${\rm dist}(M, Q_i)$ denotes the length of a minimal geodesic connecting $M$ and $Q_i$ (see, e.g.,~\cite{guigui2023riemanniangeo} for basic theory on Riemannian geometry). For the canonical metric proposed in~\cite{edelman1998}, the geodesic emanating from $Q_i$ with an initial velocity $Q_i\Delta_i$ (where $\Delta_i\in\skewm_n$) arrives at $Q_i\exp(\Delta_i)$, and it attains the length of $\frac{\sqrt{2}}{2}\|\Delta_i\|_{\F}$. Therefore, the pullback formulation of~\eqref{eq:karcher-mean} from the feasible set $M\in\so_n$ to the tangent spaces $T_{Q_i}\so_n$ is given by
\begin{equation}\label{eq:pullback-karcher-mean}
    \underset{Q_i\Delta_i\in T_{Q_i}\so_n}{\arg\min}\textstyle \frac{1}{4K}\sum_{i=1}^K \|\Delta_i\|_{\F}^2,\quad \text{s.t. $Q_i\exp(\Delta_i) = Q_j\exp(\Delta_j)$ for $i, j\leq K$.}
\end{equation}
Note that the variable $M$ in~\eqref{eq:karcher-mean} becomes the shared endpoint in the constraint $Q_i\exp(\Delta_i) = Q_j\exp(\Delta_j) = M$. Instead of searching over $M$ directly,~\eqref{eq:pullback-karcher-mean} searches over $\Delta_i\in\skewm_n$ as a preimage of $Q_i^{\T}M$. When the principal logarithm is used to select the preimage $\Delta_i = \log(Q_i^{\T}M)$, $\frac{\sqrt{2}}{2}\|\log(Q_i^{\T}M)\|_{\F}$ attains the distance between $Q_i$ and $M$. Thus,~\eqref{eq:pullback-karcher-mean} with $\Delta_i = \log(Q_i^{\T}M)$ becomes~\eqref{eq:karcher-mean}. It leads to the classic Riemannian steepest descent scheme~\cref{alg:so-karcher-mean} (see, e.g.,~\cite{moakher2002means,krakowski2007computation}) with the principal logarithm in lines~\ref{mean-alg:mean-update} and~\ref{mean-alg:data-to-mean-update}. On the other hand, adopting the diffeomorphic logarithm in lines~\ref{mean-alg:mean-update} and~\ref{mean-alg:data-to-mean-update} corresponds to the pullback formulation~\eqref{eq:pullback-karcher-mean}, which relaxes the minimal constraint on geodesics.

For the diffeomorphic-logarithm variant of~\cref{alg:so-karcher-mean}, the choices of the initial $\Delta_i$ matter. When $\Delta_i\in\mathcal{C}_*$, there are two candidates available, namely $\Delta_i^{X}$ and $\Delta_i^{Y}$ in the form of~\eqref{eq:dconnected-dpreimage}, and when $\Delta_i\notin\mathcal{C}_{*}$, there is an only candidate $\Delta_i^{\xi}$ determined by the canonical shift $\xi$. Once the initial $\Delta_i$ is chosen, the later $\Delta_i$ evolves continuously in a local diffeomorphism structure and stays on the initial D-component.

We consider a time-varying experiment with $Q_i(t)=\exp(D_i(t))\in\so_5$ for $i=1,2,3,4$, where $D_i(t)$ are continuous curves in $\skewm_5$. For every time $t\in [0,1]$, the corresponding Karcher mean $M(t) = \exp(S(t))\in \so_5$ of $Q_i(t)$ is computed using two variants of~\cref{alg:so-karcher-mean}. A natural initialization is given by 
\begin{equation}\label{eq:good-initial}
    S^{(0)}_{\text{good}}=\mathbf{0},\quad M^{(0)}_{\text{good}} = I_5, \quad \boxed{\Delta_{i,\text{ good}}^{(0)} = -D_i},
\end{equation}
which was observed to be reliable for the generated problems.

To better illustrate the stability difference between the two variants of~\cref{alg:so-karcher-mean}, we instead use the following initialization:
\begin{equation}\label{eq:poor-initial}
    S^{(0)}_{\text{bad}}=-\textstyle\frac{1}{2K}\sum_{i=1}^K D_i(t),\; M^{(0)}_{\text{bad}} = \exp(S^{(0)}_{\text{bad}}), \; \boxed{\Delta_{i,\text{bad}}^{(0)} = {\rm dlog}_{\Delta_{i,\text{ good}}^{(0)}}(Q_i(t)^{\T}M^{(0)}_{\text{bad}})},
\end{equation}
which is obtained by taking a gradient ascent step of~\eqref{eq:karcher-mean} at $I_5$. Thus,~\eqref{eq:poor-initial} is a deliberately worse initialization than~\eqref{eq:good-initial}. Note that $\Delta_{i,\text{bad}}^{(0)}$ at $M_{\rm bad}^{(0)}$ is obtained from the reliable $\Delta_{i,\text{ good}}^{(0)}=-D_i(t)$ at $Q_i(t)^{\T}M_{\rm good}^{(0)}$ via the diffeomorphic logarithm, which preserves some information from the reliable~\eqref{eq:good-initial}.

The results are reported in~\cref{fig:moving-mean}. The green curves represent the evolving data $Q_i(t)=\exp(D_i(t))$ for $i=1,2,3,4$. The red dot-dashed curve corresponds to the converged stationary point $M(t)=\exp(S(t))$ obtained using the principal logarithm. The blue dotted curve reports the converged stationary point obtained using the diffeomorphic logarithm. Both variants start with the initial conditions~\eqref{eq:poor-initial}. \Cref{fig:moving-mean} shows that the stationary point computed using the diffeomorphic logarithm evolves continuously with respect to time, whereas the principal logarithm leads to discontinuous jumps. These jumps indicate instability of the principal-logarithm variant under poor initialization, especially when the data are separated further apart, as reflected by larger canonical angles in $D_i(t)$. In contrast, with an appropriate choice of initial conditions, the diffeomorphic-logarithm variant preserves consistency by updating logarithms relative to their previous values, which leads to the observed robust performance. In practice, such an initialization can often be obtained from a priori knowledge of the data or the underlying model. This experiment highlights a key advantage of the diffeomorphic logarithm: it enables optimization procedures that remain consistent under perturbations, leading to smoother and more reliable behavior.

\begin{figure}[tbp]
    \centering
    \includegraphics[width=.7\textwidth]{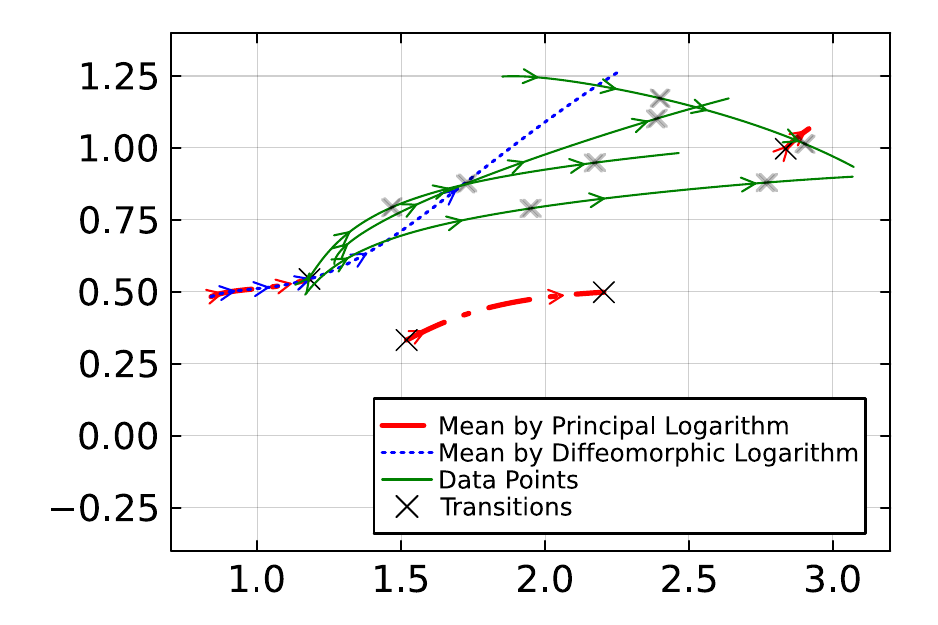}
    \caption{Converged minimizers of the moving data points}\label{fig:moving-mean}
\end{figure}

\begin{algorithm}
\caption{Karcher mean on $\so_n$ via principal or \fbox{diffeomorphic} logarithms\\{\footnotesize The diffeomorphic-logarithm variant uses the boxed terms; the principal-logarithm variant omits them.}}
\label{alg:so-karcher-mean}
\begin{algorithmic}[1]
\REQUIRE Data $\{Q_i\}_{i=1}^K \subset \so_n$, initial $S,\boxed{\{\Delta_i\}_{i=1}^K}\in\skewm_n$ where $Q_i\exp(\Delta_i) = \exp(S)$, step size $\alpha>0$, tolerance $\mathrm{AbsTol}>0$.
\ENSURE Mean $M = \exp(S)\in\so_n$.

    \STATE $M \gets \exp(S)$, $V \gets -\frac{1}{2K}\sum_{i=1}^K \log(Q_i^{\rm T}M)$ or \fbox{$V\gets -\frac{1}{2K}\sum_{i=1}^K \Delta_i$}
    \WHILE{$\|V\|_{\mathrm F} > \mathrm{AbsTol}$}
        \STATE $M \gets M\exp(\alpha V)$
        \STATE $S \gets \mathrm{log}(M)$ or \fbox{$S \gets \mathrm{dlog}_{S}(M)$}\codecomment{$\mathrm{dlog}$ denotes the diffeomorphic log.}\label{mean-alg:mean-update}

        \FOR{$i=1,\ldots,K$}
            \STATE $\Delta_i \gets \mathrm{log}(Q_i^{\T}M)$ or \fbox{$\Delta_i \gets \mathrm{dlog}_{\Delta_i}(Q_i^{\T}M)$}\codecomment{$\mathrm{dlog}$ denotes the diffeomorphic log.}\label{mean-alg:data-to-mean-update}
        \ENDFOR
        \STATE $V \gets -\frac{1}{2K}\sum_{i=1}^K \Delta_i$
    \ENDWHILE
    \RETURN $M$ and $S$
    \end{algorithmic}
\end{algorithm}

\section{Conclusion}\label{sec:conclusion}

This paper develops a geometric and computational theory for inverting the matrix exponential on the skew-symmetric matrices beyond the classical principal logarithm. A component-wise analysis of $\skewm_n \setminus \fraks$ reveals a complete inverse structure: inverse images are unique on non-special D-components, while the special component $\mathcal{C}_*$ containing the principal branch admits an explicit two-preimage characterization (\cref{prop:dconnected-dpreimage}). All remaining D-components are shown to be diffeomorphic under the exponential (\cref{thm:component-diffeomorphism}), and their structure is systematically labeled by integer shifts of $2\pi$ arising from the canonical Schur decomposition (\cref{def:canonical-decomp}). Moreover, the image of each D-component is path-connected (\cref{prop:path-connected}), and the closures of these images cover $\so_n$ (\cref{coro:dcomponent-closure}). These structural results lead to the definition of the \emph{diffeomorphic logarithm} (\cref{def:diffeomorphic-logarithm}), which extends both the principal and nearby logarithms while guaranteeing a consistent fallback preimage whenever a diffeomorphic inverse does not exist. An efficient numerical algorithm (\cref{alg:diffeomorphic-logarithm}) is proposed, demonstrating substantial computational speedups and strict preservation of skew-symmetry compared with generic matrix logarithm routines. The practical implications are illustrated through optimization on $\so_n$ (\cref{fig:moving-mean}), where diffeomorphic logarithms enable continuous evolution of the Karcher mean for time-varying data, highlighting their stability advantages in geometric computation.

\appendix

\bibliographystyle{alphaurl}
\bibliography{refs}

\end{document}